\documentclass{dcds} 
\usepackage{amsmath}
\usepackage{paralist}
\usepackage[misc]{ifsym}
\usepackage{epsfig} 
\usepackage{epstopdf} 
\usepackage[colorlinks=true]{hyperref}
\hypersetup{urlcolor=blue, citecolor=red}
\usepackage{tikz-cd}
\allowdisplaybreaks

\def\currentvolume{X}
 \def\currentissue{X}
  \def\currentyear{200X}
   \def\currentmonth{XX}
    \def\ppages{X-XX (August 17, 2026)}
     \def\DOI{10.3934/dcds.2026156}

\newtheorem{theorem}{Theorem}[section]

\newtheorem{lemma}[theorem]{Lemma}

\theoremstyle{definition}

\newtheorem{remark}[theorem]{Remark}
\newtheorem{claim}[theorem]{Claim}

\newcommand{\R}{\mathbb{R}}
\newcommand{\N}{\mathbb{N}}

\DeclareMathOperator{\shift}{shift}
\newcommand{\htop}{h_{\mathrm{top}}}
\DeclareMathOperator{\mdim}{mdim}
\newcommand{\mmdim}{\mdim_{\mathrm M}}
\newcommand{\mhdim}{\mdim_{\mathrm H}}
\DeclareMathOperator{\diam}{Diam}
\DeclareMathOperator{\sep}{sep}

\newcommand{\hdim}{\dim_{\mathrm H}}
\newcommand{\bdim}{\dim_{\mathrm M}}
\newcommand{\udim}{\overline{\dim}_{\mathrm M}\,}
\newcommand{\ldim}{\underline{\dim}_{\mathrm M}\,}
\newcommand{\umbdim}{\overline{\mdim}_{\mathrm M}}
\newcommand{\lmbdim}{\underline{\mdim}_{\mathrm M}}
\newcommand{\umhdim}{\overline{\mdim}_{\mathrm H}}
\newcommand{\lmhdim}{\underline{\mdim}_{\mathrm H}}

\DeclareMathOperator{\id}{id}
\DeclareMathOperator{\dist}{dist}
\DeclareMathOperator{\Leb}{Leb}

\title[Mean dimension theory for infinite dimensional sponges]
{Mean dimension theory for infinite dimensional Bedford--McMullen sponges} 

\author[Qiang Huo]{}


\thanks{
Q. Huo was partially supported by the National Key Research and Development Program of China (No. 2024YFA1013600), the China Postdoctoral Science Foundation (No. 2025M773065) and  Fundamental Research Funds for the Central Universities (No. WK0010250102).}

\thanks{$^*$Corresponding author: Qiang Huo}

\begin{document}
\maketitle

\centerline{\scshape
Qiang Huo$^{{\href{mailto:qianghuo@ustc.edu.cn}{\textrm{\Letter}}}*1,2}$
}


\medskip

 {\footnotesize
  \centerline{$^1$School of Mathematical Sciences, Beijing University of Posts and Telecommunications,}
  \centerline{Beijing 100876, China}
 }

\medskip
 {\footnotesize
  \centerline{$^2$Key Laboratory of Mathematics and Information Networks (Beijing University of Posts and}
  \centerline{Telecommunications), Ministry of Education, Beiiing 100876, China}
 }

\bigskip

 \centerline{(Communicated by Alejandro Maass)}


\begin{abstract}
 Tsukamoto (2022) introduced the notion of Bedford--McMullen carpet system, a subsystem of $([0,1]^{\N}\times[0,1]^{\N},\shift)$ whose metric mean dimension and mean Hausdorff dimension does not coincide in general. The aim of this paper is to develop the mean dimension theory for the Bedford--McMullen sponge system, which is a subsystem of $(([0,1]^r)^{\N},\shift)$ with arbitrary $3\leq r\in\N$. In particular, we compute the metric mean dimension and mean Hausdorff dimension of such topological dynamical systems explicitly, extending the results by Tsukamoto. The metric mean dimension is a weighted combination of the standard topological entropy, whereas the mean Hausdorff dimension is expressed in terms of weighted topological entropy. We also exhibit a special situation for which the metric mean dimension and the mean Hausdorff dimension of a sponge system coincide.

\end{abstract}


\section{Introduction}
\subsection{Background}
The well-known Bedford--McMullen carpets are a family of planar self-affine sets where all affine maps have the same linear part, corresponding to diagonal matrices. Up to now, the dimension theory for these carpets has been well studied. Bedford \cite{Bed84} and McMullen \cite{Mc84} computed the Hausdorff and Minkowski dimensions of these carpets. Moreover, McMullen pointed out that each carpet supports a unique ergodic measure with full Hausdorff dimension.
Perhaps the Bedford--McMullen family are the simplest self-affine sets. More general self-affine carpets, which allow the diagonal matrices to be non-constant, were considered by Lalley--Gatzouras \cite{GL92} and Bara\'{n}ski \cite{Bar07}. Kenyon and Peres \cite{KP96} extended the Bedford--McMullen carpets to high-dimensional counterparts, which they denoted as sponges.
They also exhibited that the Minkowski and Hausdorff dimensions of a Bedford--McMullen sponge are distinct, except in the uniform fibres case. Such phenomenon does not occur for the self-similar sets studied by Falconer \cite{Fal89}. It is worth mentioning that Das and Simmons \cite{DS17} constructed a self-affine sponge whose Hausdorff dimension is strictly bigger than the Hausdorff dimension of invariant probability measures supported on the sponge, i.e. there always exists a dimension gap, which solved a long-standing open problem in the literature negatively.

\smallskip
Motivated by the variational principle for Bedford--McMullen family (including carpets and sponges) in \cite{KP96}, i.e. the supremum of Hausdorff dimension of its invariant probability measures equals to its Hausdorff dimension, Feng and Huang \cite{FH16} introduced the notion of \textit{weighted topological entropy}.
This quantity, denoted by $\htop^{\mathbf{a}}(\pi,T)$, is concerned with a factor map
between two topological dynamical systems $(X,T)$ and $(Y,S)$, and a pair $\mathbf{a}=(a_1,a_2)\in(0,\infty)\times[0,\infty)$ of weights. The idea is to replace Bowen balls that appear in the definition of standard topological entropy \cite{Bow73}, by \textit{weighted Bowen balls} $B^{\mathbf{a}}_n(x,\varepsilon)$, which is the collection of points $y\in X$ satisfying that the $T$-orbits of $x,y$ remain $\varepsilon$-close for $\lceil a_1n \rceil$ iterates, and the $S$-orbits of $\pi(x),\pi(y)$ remain $\varepsilon$-close for $\lceil (a_1+a_2)n \rceil$ iterates.
Tsukamoto \cite{Tsu23} developed an alternative approach to weighted topological entropy, denoted by $\htop^{w}(\pi,T)$, and proved the coincidence of $\htop^{(w,1-w)}(\pi,T)$ and $\htop^{w}(\pi,T)$ by establishing variational principles for the quantity $\htop^{w}(\pi,T)$. As an application, the equality between these two quantities enables one to compute the formula for the Hausdorff dimension of Bedford--McMullen carpets. Tsukamoto's method was further extended to a sequence of topological dynamical systems of arbitrary length by Alibabaei \cite{Ali25}. We will not explain Tsukamoto's original idea tailored to the planar self-affine carpets. Instead, we review Alibabaei's definition in Section \ref{defn:Weighted topological entropy} since in this paper we focus on self-affine sponges, the higher-dimensional analogue of the carpets.

\smallskip
In his celebrated paper \cite{Gro99}, Gromov introduced the notion of \textit{mean dimension}, a new topological invariant which measures the \textit{number of parameters per unit of time required to describe a point in a system}, compared with topological entropy which measures the \textit{required number of bits per unit of time}. This invariant was systematically explored by Lindenstrauss and Weiss for discrete amenable group actions in \cite{Lin99,LW00}.
In particular, it was shown in \cite{LW00} that an upper bound of the mean dimension should be achieved by quantifying how fast the amount of entropy detected at a given resolution grows as the resolution tends to zero. This quantity describing the rate of growth is named the \textit{metric mean dimension}. Mean dimension in topological dynamics is surprisingly connected to the \textit{rate distortion dimension} \cite{KD94} in information theory via a \textit{double variational principle} established by Lindenstrauss and Tsukamoto in \cite{LT19}. More precisely, they introduced the notion of \textit{mean Hausdorff dimension}, which serves simultaneously as a finer upper bound of the mean dimension compared with the metric mean dimension and a lower bound of the rate distortion dimension over all invariant measures of the system. We review the definitions of metric mean dimension and mean Hausdorff dimension in Section \ref{defn:mean dim theory}.

\subsection{Infinite dimensional sponges and main results}\label{defn:sponge system}
In \cite{Tsu22}, Tsukamoto introduced a dynamical counterpart of some planar self-affine sets, named \textit{carpet systems}, and investigated the metric mean dimension and mean Hausdorff dimension of such carpet systems.
In this paper, we generalize the results of Tsukamoto to analogous \textit{sponge systems}, which are subsystems of $(([0,1]^r)^{\N},\shift)$ with arbitrary $3\leq r\in\N$.

Fix $r\in\N$ and let $2\leq m_1\leq \ldots \leq m_r$ be integers. Set $\mathcal{I}_l=\{0,\ldots,m_l-1\}$ for $l=1,\ldots,r$. Let $(\prod_{l=1}^{r}\mathcal{I}_l)^{\N}$ be the one-sided full shift defined on the alphabet $\prod_{l=1}^{r}\mathcal{I}_l$ equipped with the shift transformation $\sigma:(\prod_{l=1}^{r}\mathcal{I}_l)^{\N}\to (\prod_{l=1}^{r}\mathcal{I}_l)^{\N}$. By abuse of notation, we also denote by $\sigma:(\prod_{l=1}^{i}\mathcal{I}_l)^{\N}\to (\prod_{l=1}^{i}\mathcal{I}_l)^{\N}$ the shift map on $(\prod_{l=1}^{i}\mathcal{I}_l)^{\N}$. For $i=1,\ldots,r-1$, we define $p_i:\prod_{l=1}^{r}\mathcal{I}_l\to \prod_{l=1}^{i}\mathcal{I}_l$ to be the projection onto the first $i$ coordinates, i.e.
\begin{equation*}
    p_i(x_1,\ldots,x_r)=(x_1,\ldots,x_i).
\end{equation*}
Let $\pi_i:(\prod_{l=1}^{r}\mathcal{I}_l)^{\N}\to (\prod_{l=1}^{i}\mathcal{I}_l)^{\N}$ be the countable product map of $p_i$ defined by
\begin{equation}\label{projection}
    \pi_i(\mathbf{y}_1,\mathbf{y}_2,\ldots)=(p_i(\mathbf{y}_1),p_i(\mathbf{y}_2),\ldots),
\end{equation}
where $\mathbf{y}_k\in\prod_{l=1}^{r}\mathcal{I}_l$ for each $k\in\N$.

\smallskip
Let $[0,1]^{\N}$ be the Hilbert cube. Consider the product {\small $\underbrace{[0,1]^{\N}\times\cdots\times [0,1]^{\N}}_{r \text{ times}}$} equipped with a metric by
\begin{equation}\label{Tsukamoto's metric}
    d((\mathbf{x}_1,\ldots,\mathbf{x}_r),(\mathbf{y}_1,\ldots,\mathbf{y}_r))=\sum\limits_{k=1}^{\infty}2^{-k}\max\{|\mathbf{x}_{k,l}-\mathbf{y}_{k,l}|:1\leq l\leq r\}
\end{equation}
for {\small $(\mathbf{x}_1,\ldots,\mathbf{x}_r)=((\mathbf{x}_{k,1})_{k\in\N},\ldots,(\mathbf{x}_{k,r})_{k\in\N})$ and $(\mathbf{y}_1,\ldots,\mathbf{y}_r)=((\mathbf{y}_{k,1})_{k\in\N},\ldots,(\mathbf{y}_{k,r})_{k\in\N})$.} We define a shift map $\sigma:([0,1]^{\N})^r\to ([0,1]^{\N})^r$ by
\begin{equation*}
    \sigma((\mathbf{x}_{k,1})_{k\in\N},\ldots,(\mathbf{x}_{k,r})_{k\in\N})=((\mathbf{x}_{k+1,1})_{k\in\N},\ldots,(\mathbf{x}_{k+1,r})_{k\in\N}).
\end{equation*}

\smallskip
Let $\Omega\subset(\prod_{l=1}^{r}\mathcal{I}_l)^{\N}$ be a \textbf{subshift}, i.e. a closed subset with $\sigma(\Omega)\subset\Omega$. We always assume that $\Omega$ is non-empty. We define a \textbf{sponge system} $X_{\Omega}\subset([0,1]^{\N})^r$ by
\begin{equation*}
    X_{\Omega}=\left\{\left(\sum\limits_{k=1}^{\infty} \frac{\mathbf{x}_{k,1}}{m_1^k},\ldots,\sum\limits_{k=1}^{\infty}\frac{\mathbf{x}_{k,r}}{m_r^k}\right)\in([0,1]^{\N})^r: (\mathbf{x}_{k,1},\ldots,\mathbf{x}_{k,r})\in\Omega, \forall k\in\N\right\}.
\end{equation*}
Here, for $\mathbf{x}_{k,1}\in\mathcal{I}_1^{\N}\subset\ell^{\infty}$, we consider the summation $\sum_{k=1}^{\infty}\frac{\mathbf{x}_{k,1}}{m_1^k}$ in $\ell^{\infty}$. Then, $\sum_{k=1}^{\infty}\frac{\mathbf{x}_{k,1}}{m_1^k}\in[0,1]^{\N}$. The same idea is applied to $\sum_{k=1}^{\infty}\frac{\mathbf{x}_{k,i}}{m_i^k}, i=2,\ldots,r$.
In this way, we obtain a $\sigma$-invariant subset $X_{\Omega}$.
Then, $(X_{\Omega},\sigma)$ is a subsystem of $(([0,1]^{\N})^r,\sigma)$. Recall that the projection map $\pi_i$ is defined by \eqref{projection}. We consider its restriction to $\Omega$, also denoted by $\pi_i$. Then, $\pi_i(\Omega)$ is a subshift of $(\prod_{l=1}^{i}\mathcal{I}_l)^{\N}$ for each $1\leq i\leq r-1$. For convenience, sometimes we use $\pi_r$ to denote the identity map on $\Omega$.

\smallskip
For $1\leq i\leq r-1$, define the factor maps $\tau_i:\pi_{r-i+1}(\Omega)\to\pi_{r-i}(\Omega)$ by
\begin{equation*}
    \tau_i(\mathbf{x}_1,\ldots,\mathbf{x}_{r-i+1})=(\mathbf{x}_1,\ldots,\mathbf{x}_{r-i}).
\end{equation*}
Notice that $\pi_i(\Omega)=\tau_{r-i}\circ \tau_{r-i-1}\circ\cdots\circ\tau_1(\Omega)$ for each $1\leq i\leq r-1$.

\smallskip
We are now prepared to state the main result of this paper, which can be viewed as the mean dimension version of Theorem 1.2 and Proposition 1.3 (i) in \cite{KP96}.

\begin{theorem}\label{mdim of sponge system}
    Let $(X_{\Omega},\sigma,d)$ be a sponge system described as above. Set $\mathbf{a}=(a_1,\ldots,a_{r-1})$ with $a_i=\log m_{r-i}/\log m_{r-i+1}$ for each $1\leq i\leq r-1$. Then, the metric mean dimension and the mean Hausdorff dimension of $(X_{\Omega},\sigma,d)$ are given by
    \begin{equation}\label{mmdim of sponge system}
        \mmdim(X_{\Omega},\sigma,d)=\frac{\htop(\Omega,\sigma)}{\log m_r}+\sum\limits_{i=1}^{r-1}\left(\frac{1}{\log m_{i}}-\frac{1}{\log m_{i+1}}\right)\htop(\pi_i(\Omega),\sigma)
    \end{equation}
    and
    \begin{equation}\label{mhdim of sponge system}
        \mhdim(X_{\Omega},\sigma,d)=\frac{\htop^{\mathbf{a}}(\{(\pi_i(\Omega),\sigma)\}_{i=1}^r,\{\tau_i\}_{i=1}^{r-1})}{\log m_1}
    \end{equation}
    respectively, where $\htop(\pi_i(\Omega),\sigma)$ denotes the topological entropy of $(\pi_i(\Omega),\sigma)$ for each $1\leq i\leq r$ and $\htop^{\mathbf{a}}(\{(\pi_i(\Omega),\sigma)\}_{i=1}^r,\{\tau_i\}_{i=1}^{r-1})$ denotes the weighted topological entropy of $\mathbf{a}$-exponent with respect to factor maps $\tau_i:(\pi_{r-i+1}(\Omega),\sigma)\to(\pi_{r-i}(\Omega),\sigma),\\ 1\leq i\leq r-1$.
\end{theorem}

The proof of Theorem \ref{mdim of sponge system} was inspired by \cite{Tsu22,KP96}. The notion of \textit{approximate cubes} plays a prominent role in the computation of the metric mean dimension. For the mean Hausdorff dimension, we express it in terms of the weighted topological entropy for a sequence of topological dynamical systems, we use a combinatorial lemma (see Lemma \ref{KP96,Lemma4.1}) by Kenyon and Peres \cite{KP96}, and we also use some tools from probability theory (see Lemma \ref{Lemma5.9}) and geometric measure theory (see Lemma \ref{Hausdorff dim estimate}).

\subsection{Notation and organization of the paper.}
    In Section \ref{Preliminaries}, we recall the definitions of Minkowski dimension, Hausdorff dimension and their dynamical counterparts, and Alibabaei's approach to define weighted topological entropy for higher dimensions.
    Section \ref{Proofs} is devoted to the results about sponge systems, including the proof of Theorem \ref{mdim of sponge system}. We also discuss a special situation in Lemma \ref{coincidence} for which the metric mean dimension and the mean Hausdorff dimension of a sponge system coincide.

\smallskip
    In this paper, we denote the set of positive integers by $\N$ and the set of non-negative integers by $\N_0$. For $a\in\R$, $\lceil a\rceil$ denotes the smallest integer not less than $a$ whereas $\lfloor a\rfloor$ denotes the biggest integer not more than $a$. For a random variable $\eta$, denote its expectation and variance by $\mathbb{E}(\eta)$ and $\mathbb{V}(\eta):=\mathbb{E}((\eta-\mathbb{E}(\eta))^2)$, respectively. By $\diam U$ we denote the diameter of a set $U$. $|U|$ is the cardinality of a set $U$. We use boldface to denote vectors and a subscript to denote each coordinate of a vector. For vectors $\mathbf{x,y}\in[0,1]^n$ with $\mathbf{x}=(x_1,\ldots,x_n),\mathbf{y}=(y_1,\ldots,y_n)$, the $\ell^{\infty}$-distance is defined by
    \begin{equation*}
        \|\mathbf{x-y}\|_{\infty}=\max_{1\leq k\leq n}|x_k-y_k|.
    \end{equation*}

\section{Preliminaries}\label{Preliminaries}
We split this section into three pieces, where we review basic definitions of fractal dimensions, mean dimensions, and weighted topological entropy, respectively.
\subsection{Fractal dimensions}
Let $(X,d)$ be a compact metrizable space. For $\varepsilon>0$, we define the $\varepsilon$-\textbf{covering number} $\#(X,d,\varepsilon)$ as the minimum $n\geq 1$ such that there exists an open cover $\{U_1,\ldots,U_n\}$ of $X$ satisfying that $\diam U_i<\varepsilon$ for each $1\leq i\leq n$.
We define the \textbf{upper and lower Minkowski dimensions} (also called \textbf{box-counting dimensions}) of $(X,d)$ by
\begin{equation*}
    \udim(X,d)=\limsup\limits_{\varepsilon\to0}\bdim(X,d,\varepsilon), \ldim(X,d)=\liminf\limits_{\varepsilon\to0}\bdim(X,d,\varepsilon),
\end{equation*}
where
\begin{equation*}
    \bdim(X,d,\varepsilon)=\frac{\log\#(X,d,\varepsilon)}{\log(1/\varepsilon)}
\end{equation*}
denotes the $\varepsilon$-\textbf{scale Minkowski dimension} of $(X,d)$. If these two values coincide, the common value is called the \textbf{Minkowski dimension} of $(X,d)$. Remark that if we define the $\varepsilon$-\textbf{separating number} $\#_{\sep}(X,d,\varepsilon)$ as the maximum $n\geq 1$ such that there exist $x_1,\ldots,x_n\in X$ satisfying $d(x_i,x_j)\geq\varepsilon$ for all distinct $i,j$, then we can replace the covering numbers in the definition of the Minkowski dimension with separating numbers since
\begin{equation*}
    \#_{\sep}(X,d,\varepsilon)\leq \#(X,d,\varepsilon) \leq \#_{\sep}(X,d,\varepsilon/4).
\end{equation*}

For $s\geq 0$ and $\varepsilon>0$, we define $\mathcal{H}_{\varepsilon}^s(X,d)$ by
\begin{equation*}
    \mathcal{H}_{\varepsilon}^s(X,d)=\inf\left\{ \sum_{i=1}^{\infty}(\diam E_i)^s: X=\bigcup_{i=1}^{\infty} E_i \text{ with }  \diam E_i<\varepsilon, \forall i\geq 1\right\}.
\end{equation*}
We use the convention that $(\diam\emptyset)^s=0$ for all $s\geq 0$ and $0^0=1$. We define the \textbf{Hausdorff dimension} of $(X,d)$ by
\begin{equation*}
    \hdim(X,d)=\lim_{\varepsilon\to0} \hdim(X,d,\varepsilon),
\end{equation*}
where $\hdim(X,d,\varepsilon)=\sup\{s\geq 0: \mathcal{H}_{\varepsilon}^s(X,d)\geq 1\}$ denotes the $\varepsilon$-\textbf{scale Hausdorff dimension} of $(X,d)$.

\subsection{Mean dimension theory}\label{defn:mean dim theory}
By a \textbf{topological dynamical system} (\textit{TDS} for short), we understand a pair $(X,T)$ where $X$ is a compact metrizable space and $T:X\to X$ is a continuous map.
Let $(X,T)$ be a \textit{TDS} equipped with a metric $d$ on $X$. For each $N\in\N$, we define the Bowen metric $d_N$ on $X$ by
\begin{equation*}
    d_N(x,y)=\max_{0\leq n<N}d(T^nx,T^ny).
\end{equation*}
The \textbf{topological entropy} of $(X,T)$ is defined by
\begin{equation*}
    \htop(X,T)=\lim_{\varepsilon\to0}\left(\lim_{N\to\infty}\frac{\log \#(X,d_N,\varepsilon)}{N}\right).
\end{equation*}
The \textbf{upper and lower metric mean dimensions} of $(X,T,d)$ are defined by
\begin{equation*}
    \begin{split}
        &\umbdim(X,T,d)=\limsup_{\varepsilon\to0}\left(\lim_{N\to\infty}\frac{\bdim(X,d_N,\varepsilon)}{N}\right),\\
        &\lmbdim(X,T,d)=\liminf_{\varepsilon\to0}\left(\lim_{N\to\infty}\frac{\bdim(X,d_N,\varepsilon)}{N}\right).
    \end{split}
\end{equation*}
The limit over $N$ exists since $\log\#(X,d_N,\varepsilon)$ is sub-additive in $N$. If these two values coincide, the common value is called the \textbf{metric mean dimension} of $(X,T,d)$, denoted by $\mmdim(X,T,d)$. We define the \textbf{upper and lower mean Hausdorff dimensions} of $(X,T,d)$ by
\begin{equation*}
    \begin{split}
        &\umhdim(X,T,d)=\lim_{\varepsilon\to0}\left(\limsup_{N\to\infty}\frac{\hdim(X,d_N,\varepsilon)}{N}\right),\\
        &\lmhdim(X,T,d)=\lim_{\varepsilon\to0}\left(\liminf_{N\to\infty}\frac{\hdim(X,d_N,\varepsilon)}{N}\right).
    \end{split}
\end{equation*}
If these two values coincide, the common value is called the \textbf{mean Hausdorff dimension} of $(X,T,d)$ and denoted by $\mhdim(X,T,d)$.
Mean Hausdorff dimension and upper and lower metric mean dimension are invariants for bi-Lipschitz isomorphisms, and they are all metric-dependent.
Just as the Hausdorff dimension is never bigger than the Minkowski dimension, one has the following relationship (cf. \cite[Proposition 2.1]{Tsu22}):
\begin{equation*}
    \lmhdim(X,T,d)\leq  \umhdim(X,T,d)\leq \lmbdim(X,T,d)\leq \umbdim(X,T,d).
\end{equation*}

\subsection{Weighted topological entropy}\label{defn:Weighted topological entropy}
We follow the notation set up by Alibabaei in \cite{Ali25}.
Consider a sequence of \textit{TDS}s $(X_i,T_i),i=1,\ldots,r$ and factor maps $\tau_i:X_i\to X_{i+1},i=1,\ldots,r-1$.
\begin{equation*}
	\begin{tikzcd}
		(X_1,T_1) \arrow{r}{\tau_{1}} \arrow[bend left=30]{r}{\pi_{r-1}} \arrow[bend left=47]{rr}{\pi_{r-i+1}} \arrow[bend left=52]{rrr}{\pi_1} & (X_2,T_2) \arrow{r}{\tau_2}  & \cdots (X_i,T_i) \cdots \arrow{r}{\tau_{r-1}} &  (X_r,T_r)
	\end{tikzcd}
\end{equation*}
Equip $X_i$ with a metric $d^{(i)}$.
Let $\mathbf{a}=(a_1,a_2,\ldots,a_{r-1})$ with $0\leq a_i\leq 1$ for each $i$. For $N\in\N$, we define the Bowen metric $d_N^{(i)}$
on $X_i$ by
\begin{equation*}
    d_N^{(i)}(x,y)=\max_{0\leq n<N}d^{(i)}(T_i^nx,T_i^ny).
\end{equation*}
We inductively define a quantity $\#^{\mathbf{a}}(\Omega,N,\varepsilon)$ for $\Omega\subset X_i$. For $\Omega\subset X_i$, set
\begin{equation*}
    \#_1^{\mathbf{a}}(\Omega,N,\varepsilon)=\min \left\{\begin{array}{c|c}
n\in\N & \begin{array}{c}
 \text { There exists an open cover } \{U_j\}_{j=1}^n \text { of } \Omega  \\
 \text { with } \diam(U_j, d_N^{(1)})<\varepsilon \text { for all } 1 \leq j \leq n
\end{array}
\end{array}\right\} .
\end{equation*}
Let $\Omega\subset X_{i+1}$. Suppose that $\#_i^{\mathbf{a}}$ is already defined. We set
\begin{equation*}
{\small \begin{split}
    & \#_{i+1}^{\mathbf{a}}(\Omega,N,\varepsilon) \\
    &=\min \left\{\begin{array}{c|c}
\sum\limits_{j=1}^n(\#_{i}^{\mathbf{a}}(\tau_i^{-1}(U_j),N,\varepsilon))^{a_i} & \begin{array}{c}
 n\in\N, \{U_j\}_{j=1}^n \text { is an open cover of } \Omega  \\
 \text { with } \diam(U_j, d_N^{(i+1)})<\varepsilon \text { for all } 1 \leq j \leq n
\end{array}
\end{array}\right\} .
\end{split}}
\end{equation*}
We denote the \textbf{weighted topological entropy of} $\mathbf{a}$-\textbf{exponent} by\\ $\htop^{\mathbf{a}}(\{(X_i,T_i)\}_{i=1}^r,\{\tau_i\}_{i=1}^{r-1})$ and define it by
\begin{equation*}
    \htop^{\mathbf{a}}(\{(X_i,T_i)\}_{i=1}^r,\{\tau_i\}_{i=1}^{r-1})=\lim_{\varepsilon\to0}\left(\lim_{N\to\infty}\frac{\log\#_r^{\mathbf{a}}(X_r,N,\varepsilon)}{N}\right).
\end{equation*}
The limit exists since $\log\#_r^{\mathbf{a}}(X_r,N,\varepsilon)$ is sub-additive in $N$ and non-decreasing as $\varepsilon$ tends to $0$. Note that in the above definition, we can use closed covers instead of open covers, which does not change the value.

\smallskip
We refer interested readers to \cite{KP96a,Ali24} for the calculation of the weighted topological entropy of some sofic subshifts, which are defined from graph-directed self-affine sets.

\smallskip
From $\mathbf{a}=(a_1,\ldots,a_{r-1})\in\R^{r-1}$, we define a probability vector $\mathbf{w_a}=(w_1,\ldots,w_r)\\\in\R^r$ by
\begin{equation}\label{wa}
    \left\{
\begin{array}{l}
w_1=a_1a_2a_3\cdots a_{r-1}, \\
w_2=(1-a_1)a_2a_3\cdots a_{r-1}, \\
w_3=(1-a_2)a_3\cdots a_{r-1},\\
\vdots\\
w_{r-1}=(1-a_{r-2})a_{r-1},\\
w_r=1-a_{r-1}.
\end{array}
\right.
\end{equation}
The weighted topological entropy satisfies the following variational principle, see \cite[Theorem 1.2]{Ali25}:
\begin{equation}\label{variational principle}
    \htop^{\mathbf{a}}(\{(X_i,T_i)\}_{i=1}^r,\{\tau_i\}_{i=1}^{r-1})=\sup_{\mu\in\mathcal{M}^{T_1}(X_1)}\left(\sum\limits_{i=1}^{r}w_i h_{(\pi_{r-i+1})_*\mu}(T_i)\right),
\end{equation}
where $\pi_i$ is defined by
\begin{equation*}
    \pi_r=\id_{X_1}: X_1\to X_1,
\end{equation*}
\begin{equation*}
    \pi_i=\tau_{r-i}\circ\tau_{r-i-1}\circ\cdots\circ\tau_1 : X_1\to X_{r-i+1}, 1\leq i\leq r-1,
\end{equation*}
$\mathcal{M}^{T_1}(X_1)$ is the set of $T_1$-invariant Borel probability measures on $X_1$, $(\pi_{r-i+1})_*\mu$ is the push-forward measure of $\mu$ by $\pi_{r-i+1}$ on $X_i$, and $h_{(\pi_{r-i+1})_*\mu}(T_i)$ is the Kolmogorov-Sinai entropy of the measure preserving system $(X_i,T_i,(\pi_{r-i+1})_*\mu)$ for each $i$.

\begin{remark}
    Given $\mathbf{a}=\left(\log m_{r-1}/\log m_r,\log m_{r-2}/\log m_{r-1},\ldots,\log m_1/\log m_2\right)$, the probability vector $\mathbf{w_a}$ defined by \eqref{wa} equals
    \begin{equation*}
        \left( \frac{\log m_1}{\log m_r}, \frac{\log m_1}{\log m_{r-1}}-\frac{\log m_1}{\log m_r},\ldots, \frac{\log m_1}{\log m_2}-\frac{\log m_1}{\log m_3}, 1-\frac{\log m_1}{\log m_2} \right).
    \end{equation*}
    Combining formulae \eqref{mhdim of sponge system} and \eqref{variational principle}, we obtain an alternative expression of the mean Hausdorff dimension as follows:
    {\small \begin{equation*}
        \mhdim(X_{\Omega},\sigma,d)= \sup_{\mu\in\mathcal{M}^{\sigma}(\Omega)}\left( \frac{h_{\mu}(\Omega,\sigma)}{\log m_r}+\sum\limits_{i=1}^{r-1}\left(\frac{1}{\log m_{i}}-\frac{1}{\log m_{i+1}}\right)h_{(\pi_i)_{*}\mu}(\pi_i(\Omega),\sigma) \right).
    \end{equation*}}Moreover, using the standard variational principle for topological entropy (cf. \cite[Chapter 8]{Wal82}), \eqref{mmdim of sponge system} yields that
    \begin{equation*}
    \begin{split}
        \mmdim(X_{\Omega},\sigma,d)&=\sup_{\mu\in\mathcal{M}^{\sigma}(\Omega)}\frac{h_{\mu}(\Omega,\sigma)}{\log m_r}\\
        &\quad +\sum\limits_{i=1}^{r-1}\left(\frac{1}{\log m_{i}}-\frac{1}{\log m_{i+1}}\right)\sup_{\nu_i\in\mathcal{M}^{\sigma}(\pi_i(\Omega))}h_{\nu_i}(\pi_i(\Omega),\sigma).
    \end{split}
    \end{equation*}
    Observe that the difference between the formulae for the mean Hausdorff dimension and the metric mean dimension of a sponge system is whether we take supremum simultaneously or separately for these terms involving measure-theoretic entropy, i.e. $h_{\mu}(\Omega,\sigma)$ and $h_{\nu_i}(\pi_i(\Omega),\sigma), i=1,\ldots,r-1$. In particular, if $(X_{\Omega},\sigma)$ is a carpet system, i.e. $r=2$, then Theorem \ref{mdim of sponge system} reduces to \cite[Theorem 5.3]{Tsu22}.
\end{remark}

\section{Proof of Theorem \ref{mdim of sponge system}}\label{Proofs}
\subsection{Calculation of metric mean dimension}\label{Calculation of metric mean dimension}
In this section, we prove \eqref{mmdim of sponge system} in Theorem \ref{mdim of sponge system}. Fix $N\in\N$. For each $1\leq i\leq r$, denote by $\pi_i(\Omega)\vert_N$ the image under the projection
\begin{equation*}
    (\prod_{l=1}^{i}\mathcal{I}_l)^{\N}\to \prod_{l=1}^{i}\mathcal{I}_l^N,
    ((\mathbf{x}_{k,1})_{k\in\N},\ldots,(\mathbf{x}_{k,i})_{k\in\N}) \mapsto ((\mathbf{x}_{k,1})_{1\leq k\leq N},\ldots,(\mathbf{x}_{k,i})_{1\leq k\leq N}).
\end{equation*}
We also define $X_{\Omega}\vert_N$ as the image of $X_{\Omega}$ under the projection
\begin{equation*}
    ([0,1]^{\N})^r \to [0,1]^{rN},
    ((\mathbf{y}_{k,1})_{k\in\N},\ldots,(\mathbf{y}_{k,r})_{k\in\N}) \mapsto ((\mathbf{y}_{k,1})_{1\leq k\leq N},\ldots,(\mathbf{y}_{k,r})_{1\leq k\leq N}).
\end{equation*}
Then,
\begin{equation*}
    X_{\Omega}\vert_N=\left\{\left(\sum\limits_{k=1}^{\infty} \frac{\mathbf{x}_{k,1}}{m_1^k},\ldots,\sum\limits_{k=1}^{\infty}\frac{\mathbf{x}_{k,r}}{m_r^k}\right)\in[0,1]^{rN} \middle|\,  (\mathbf{x}_{k,1},\ldots,\mathbf{x}_{k,r})\in\Omega\vert_N, \forall k\in\N\right\}.
\end{equation*}
Fix $M\in\N$. Let $L_i(M)$ be the unique integer satisfying
\begin{equation*}
    m_i^{-L_i(M)-1}<m_1^{-M} \leq m_i^{-L_i(M)}
\end{equation*}
for $i=1,\ldots,r$. Note that $L_1(M)=M$. In particular, we have
\begin{equation}\label{Li(M)}
    M\frac{\log m_1}{\log m_i}-1 < L_i(M) \leq M\frac{\log m_1}{\log m_i},
\end{equation}
i.e. {\small $L_i(M)=\lfloor M\log m_1/\log m_i \rfloor$. Fix a point $(\mathbf{x}_1,\ldots,\mathbf{x}_r)=((\mathbf{x}_{k,1})_{k\in\N},\ldots,(\mathbf{x}_{k,r})_{k\in\N})\in (\Omega\vert_N)^{\N}$} with $\mathbf{x}_{k,l}\in\mathcal{I}_l^N$ and $(\mathbf{x}_{k,1},\ldots,\mathbf{x}_{k,r})\in\Omega\vert_N$ for each $k\in\N$. We denote the \textbf{approximate cube} of level $M\in\N$ centred at $(\mathbf{x}_1,\ldots,\mathbf{x}_r)$ by $Q_{N,M}(\mathbf{x}_1,\ldots,\mathbf{x}_r)$, and define it by
 \begin{equation} \label{approximate cube}
  \begin{split}  
  &Q_{N,M}(\mathbf{x}_1,\ldots,\mathbf{x}_r) \\
  & =
   \left\{\left(\sum\limits_{k=1}^{\infty} \frac{\mathbf{y}_{k,1}}{m_1^k},\ldots,\sum\limits_{k=1}^{\infty}\frac{\mathbf{y}_{k,r}}{m_r^k}\right)  \middle|\,
    \parbox{2.5in}{\centering $(\mathbf{y}_{k,1},\ldots,\mathbf{y}_{k,r})\in\Omega\vert_N$, $\forall k\in\N$ with \\
    $\mathbf{y}_{k,i}=\mathbf{x}_{k,i}$ for all $1\leq k\leq L_i(M)$ \\ and $1\leq i \leq r$} \right\}.
\end{split}\end{equation}

\smallskip
The set $Q_{N,M}(\mathbf{x}_1,\ldots,\mathbf{x}_r)$ only relies on the coordinates $\mathbf{x}_{1,1},\ldots,\mathbf{x}_{L_1(M),1},\ldots, \\ \mathbf{x}_{1,r},\ldots, \mathbf{x}_{L_r(M),r}$.
Note that by \eqref{Li(M)} we immediately have
\begin{equation*}
    \diam(Q_{N,M}(\mathbf{x}_1,\ldots,\mathbf{x}_r),\|\cdot\|_{\infty})\leq\max\{m_i^{-L_i(M)}: 1\leq i\leq r\}\leq m_rm_1^{-M},
\end{equation*}
where $\|\cdot\|_{\infty}$ denotes the $\ell^{\infty}$-norm on $X_{\Omega}\vert_N\subset\R^{rN}$.

\smallskip
The following lemma reveals that we can compute the metric mean dimension of a sponge system equivalently in a way which completely avoids the use of both the shift map and the metric \eqref{Tsukamoto's metric}. Instead, we only need to consider the infinity norm on finite-dimensional cubes and the projection of the subshift. Similar ideas have appeared in \cite[Proposition A.3 \& Lemma A.4]{mmdimcompress} and \cite[Lemma 5.5]{Tsu22}.

\begin{lemma}\label{equivalent metric}
    Let $\varepsilon>0$. For any $N\in\N$, we have
    \begin{equation*}
        \#_{\sep}(X_{\Omega}\vert_N,\|\cdot\|_{\infty},\varepsilon) \leq \#_{\sep}(X_{\Omega},d_N,\varepsilon).
    \end{equation*}
    If $N_0\in\N$ satisfies $\sum_{n\geq N_0}2^{-n}<\varepsilon/2$, then
    \begin{equation*}
        \#(X_{\Omega},d_N,\varepsilon) \leq \#(X_{\Omega}\vert_{N+N_0},\|\cdot\|_{\infty},\frac{\varepsilon}{2}).
    \end{equation*}
\end{lemma}

\begin{proof}
    Fix $N\in\N$. Observe that for $((\mathbf{x}_{k,1})_{k\in\N},\ldots,(\mathbf{x}_{k,r})_{k\in\N}),((\mathbf{y}_{k,1})_{k\in\N},\ldots,\\(\mathbf{y}_{k,r})_{k\in\N})\in X_{\Omega}$, the inequality
    $\|((\mathbf{x}_{k,1})_{1\leq k\leq N},\ldots,(\mathbf{x}_{k,r})_{1\leq k\leq N})-((\mathbf{y}_{k,1})_{1\leq k\leq N},\\ \ldots,(\mathbf{y}_{k,r})_{1\leq k\leq N})\|_{\infty}\geq\varepsilon$ implies that
    \begin{equation*}
        d_N(((\mathbf{x}_{k,1})_{k\in\N},\ldots,(\mathbf{x}_{k,r})_{k\in\N}),((\mathbf{y}_{k,1})_{k\in\N},\ldots,(\mathbf{y}_{k,r})_{k\in\N}))\geq\varepsilon
    \end{equation*}
    and hence the first inequality holds. The second inequality follows from
    \begin{equation*}
    \begin{split}
        &d_N(((\mathbf{x}_{k,1})_{k\in\N},\ldots,(\mathbf{x}_{k,r})_{k\in\N}),((\mathbf{y}_{k,1})_{k\in\N},\ldots,(\mathbf{y}_{k,r})_{k\in\N}))\\
        &< \|((\mathbf{x}_{k,1})_{1\leq k\leq N+N_0},\ldots,(\mathbf{x}_{k,r})_{1\leq k\leq N+N_0})\\
        &\quad -((\mathbf{y}_{k,1})_{1\leq k\leq N+N_0},\ldots,(\mathbf{y}_{k,r})_{1\leq k\leq N+N_0})\|_{\infty}+\frac{\varepsilon}{2}.
    \end{split}
    \end{equation*}
\end{proof}

\begin{lemma}\label{upper and lower bound of mmdim}
    Fix $N,M\in\N$. Then, we have
    \begin{equation*}
        \begin{split}
            \#(X_{\Omega}\vert_N, \|\cdot\|_{\infty}, m_rm_1^{-M}) &\leq |\Omega\vert_N|^{L_r(M)}\cdot\prod_{i=1}^{r-1}|\pi_i(\Omega)\vert_N|^{L_i(M)-L_{i+1}(M)},\\
            \#_{\sep}(X_{\Omega}\vert_N, \|\cdot\|_{\infty}, m_1^{-M}) &\geq |\Omega\vert_N|^{L_r(M)}\cdot\prod_{i=1}^{r-1}|\pi_i(\Omega)\vert_N|^{L_i(M)-L_{i+1}(M)}.
        \end{split}
    \end{equation*}
\end{lemma}
\begin{proof}
    Recall that $\diam(Q_{N,M}(\mathbf{x}_1,\ldots,\mathbf{x}_r),\|\cdot\|_{\infty}) \leq m_rm_1^{-M}$. The first inequality follows from
    \begin{equation*}
    X_{\Omega}\vert_N  =
   \bigcup\left\{Q_{N,M}(\mathbf{x}_1,\ldots,\mathbf{x}_r)  \middle|\,
    \parbox{2.5in}{\centering $(\mathbf{x}_{k,1},\ldots,\mathbf{x}_{k,r})\in\Omega\vert_N$, for each $1\leq k\leq L_r(M)$ and \\
    $(\mathbf{x}_{k,1},\ldots,\mathbf{x}_{k,i})\in \pi_i(\Omega)\vert_N$ for each $L_{i+1}(M)+1\leq k\leq L_i(M)$ \\ and $1\leq i\leq r-1$} \right\}.
    \end{equation*}
    We now prove the second inequality. Fix a point $(\mathbf{y}_1,\ldots,\mathbf{y}_r)\in\Omega\vert_N$. Fix $1\leq i\leq r-1$. For each $(\mathbf{v}_1,\ldots,\mathbf{v}_i)\in\pi_i(\Omega)\vert_N$, we choose $\mathbf{s}_i(\mathbf{v}_1,\ldots,\mathbf{v}_i)\in\mathcal{I}_{i+1}^N$ such that $(\mathbf{v}_1,\ldots,\mathbf{v}_i,\mathbf{s}_i(\mathbf{v}_1,\ldots,\mathbf{v}_i))\in\pi_{i+1}(\Omega)\vert_N$. For $(\mathbf{x}_{k,1},\ldots,\mathbf{x}_{k,r})\in\Omega\vert_N (1\leq k\leq L_r(M))$ satisfying that $(\mathbf{x}_{k,1},\ldots,\mathbf{x}_{k,i})\in \pi_i(\Omega)\vert_N$ for each $L_{i+1}(M)+1\leq k\leq L_i(M)$ and $1\leq i\leq r-1$, we choose a point
    \begin{equation*}
       {\SMALL  \begin{split}
            &p(\mathbf{x}_{1,1},\ldots,\mathbf{x}_{L_1(M),1},\ldots,\mathbf{x}_{1,r},\ldots,\mathbf{x}_{L_r(M),r})\\[2mm]
            &=\left( \sum_{k=1}^{L_1(M)}\frac{\mathbf{x}_{k,1}}{m_1^k}+ \sum_{k=L_1(M)+1}^{\infty}\frac{\mathbf{y}_{1}}{m_1^k},
             \sum_{k=1}^{L_2(M)}\frac{\mathbf{x}_{k,2}}{m_2^k}+ \sum_{k=L_2(M)+1}^{L_1(M)}\frac{\mathbf{s}_1(\mathbf{x}_{k,1})}{m_2^k}+ \sum_{k=L_1(M)+1}^{\infty}\frac{\mathbf{y}_{2}}{m_2^k},
             \right.\\
             &\quad \sum_{k=1}^{L_3(M)}\frac{\mathbf{x}_{k,3}}{m_3^k} + \sum_{k=L_3(M)+1}^{L_2(M)}\frac{\mathbf{s}_2(\mathbf{x}_{k,1},\mathbf{x}_{k,2})}{m_3^k} +
             \sum_{k=L_2(M)+1}^{L_1(M)}\frac{\mathbf{s}_2(\mathbf{x}_{k,1},\mathbf{s}_1(\mathbf{x}_{k,1}))}{m_3^k}
             +\sum_{k=L_1(M)+1}^{\infty}\frac{\mathbf{y}_{3}}{m_3^k},\ldots,\\[2mm]
            &\quad \left.\sum_{k=1}^{L_r(M)}\frac{\mathbf{x}_{k,r}}{m_r^k} + \sum_{i=1}^{r-1}\sum_{k=L_{i+1}(M)+1}^{L_i(M)}\frac{\mathbf{s}_{r-1}(\mathbf{x}_{k,1},\ldots,\mathbf{x}_{k,i},\mathbf{s}_{i+1}(\mathbf{x}_{k,1},\ldots,\mathbf{x}_{k,i}),\ldots)}{m_r^k} +
            \sum_{k=L_1(M)+1}^{\infty}\frac{\mathbf{y}_{r}}{m_r^k} \right),
        \end{split}}
    \end{equation*}
    which lies in $X_{\Omega}\vert_N$. Furthermore, for two distinct points $(\mathbf{x}_{1,1},\ldots,\mathbf{x}_{L_1(M),1},\ldots,\\ \mathbf{x}_{1,r},\ldots,\mathbf{x}_{L_r(M),r})$ and $ (\mathbf{x}'_{1,1},\ldots,\mathbf{x}'_{L_1(M),1},\ldots,\mathbf{x}'_{1,r},\ldots,\mathbf{x}'_{L_r(M),r})$, we have
    \begin{equation*}
         \begin{split}
            &\|p(\mathbf{x}_{1,1},\ldots,\mathbf{x}_{L_1(M),1},\ldots,\mathbf{x}_{1,r},\ldots,\mathbf{x}_{L_r(M),r})\\
            &-p(\mathbf{x}'_{1,1},\ldots,\mathbf{x}'_{L_1(M),1},\ldots,\mathbf{x}'_{1,r},\ldots,\mathbf{x}'_{L_r(M),r})\|_{\infty}\\
            &\geq \min\{m_i^{-L_i(M)}: 1\leq i\leq r\}\geq m_1^{-M}.
        \end{split}
    \end{equation*}
    Therefore, the set
    \begin{equation*}
        {\small \begin{split}
            &\left\{p(\mathbf{x}_{1,1},\ldots,\mathbf{x}_{L_1(M),1},\ldots,\mathbf{x}_{1,r},\ldots,\mathbf{x}_{L_r(M),r}): (\mathbf{x}_{k,1},\ldots,\mathbf{x}_{k,r})\in\Omega\vert_N \text{ for } 1\leq k\leq L_r(M),\right.\\[2mm]
             &\left. (\mathbf{x}_{k,1},\ldots,\mathbf{x}_{k,i})\in \pi_i(\Omega)\vert_N \text{ for each } L_{i+1}(M)+1\leq k\leq L_i(M) \text{ and } 1\leq i\leq r-1 \right\}
        \end{split}}
    \end{equation*}
    is $m_1^{-M}$-separated with respect to the $\ell^{\infty}$-distance, which yields the second inequality.
\end{proof}

\begin{proof}[Proof of formula \eqref{mmdim of sponge system} in Theorem \ref{mdim of sponge system}]
    Combining Lemma \ref{equivalent metric}, Lemma \ref{upper and lower bound of mmdim}, and \eqref{Li(M)}, we have
    \begin{equation*}
       {\small  \begin{split}
           & \mmdim(X_{\Omega},\sigma,d)\\
            &=\lim\limits_{M\to\infty}\lim\limits_{N\to\infty}\frac{\log(|\Omega\vert_N|^{L_r(M)}\cdot\prod_{i=1}^{r-1}|\pi_i(\Omega)\vert_N|^{L_i(M)-L_{i+1}(M)})}{NM\log m_1} \\
            &=\lim\limits_{M\to\infty}\frac{1}{M\log m_1}\lim\limits_{N\to\infty}\frac{L_r(M)\log |\Omega\vert_N| + \sum_{i=1}^{r-1}(L_i(M)-L_{i+1}(M))\log |\pi_i(\Omega)\vert_N|}{N} \\
            &=\lim\limits_{M\to\infty}\frac{L_r(M)\htop(\Omega,\sigma)+ \sum_{i=1}^{r-1}(L_i(M)-L_{i+1}(M)) \htop(\pi_i(\Omega),\sigma)}{M\log m_1} \\
            &=\frac{\htop(\Omega,\sigma)}{\log m_r}+\sum\limits_{i=1}^{r-1}\left(\frac{1}{\log m_{i}}-\frac{1}{\log m_{i+1}}\right)\htop(\pi_i(\Omega),\sigma).
        \end{split}}
    \end{equation*}
\end{proof}

\subsection{Calculation of weighted topological entropy}
As we emphasized in the introduction, the mean Hausdorff of a sponge system is expressed via the weighted topological entropy. For this reason, in this section, we first give an explicit formula to calculate the weighted topological entropy.

\smallskip
Let $\mathcal{I}_{l},l=1,\ldots,r$ be finite sets, and let $((\prod_{l=1}^i\mathcal{I}_{l})^{\N},\sigma)$ be the full shift on the alphabet $\prod_{l=1}^i\mathcal{I}_{l}$ for each $1\leq i\leq r$. For $1\leq i\leq r-1$ , let $\tau_i:(\prod_{l=1}^{r-i+1}\mathcal{I}_{l})^{\N}\to(\prod_{l=1}^{r-i}\mathcal{I}_{l})^{\N}$ be the natural projection defined by
\begin{equation*}
    \tau_i(\mathbf{x}_1,\ldots,\mathbf{x}_{r-i+1})=(\mathbf{x}_1,\ldots,\mathbf{x}_{r-i})
\end{equation*}
and let $\pi_i=\tau_{r-i}\circ\tau_{r-i-1}\circ\cdots\circ\tau_1$. Let $\Omega\subset(\prod_{l=1}^r\mathcal{I}_{l})^{\N}$ be a subshift. Then, $\pi_i(\Omega)$ is a subshift of $(\prod_{l=1}^i\mathcal{I}_{l})^{\N}$ for $1\leq i\leq r-1$.
For convenience, we use $\pi_r$ to denote the identity map on $\Omega$.
For $N\in\N$, we define $\pi_i(\Omega)\vert_N\subset\prod_{l=1}^i\mathcal{I}_{l}^N$ as the image of the projection of $\pi_i(\Omega)$ to the first $N$ coordinates. We denote by $\tau_i\vert_N:\pi_{r-i+1}(\Omega)\vert_N\to\pi_{r-i}(\Omega)\vert_N$ the natural projection map.

\begin{lemma}\label{Formula for the weighted topological entropy}
    Let $\mathbf{a}=(a_1,\ldots,a_{r-1})$ with $0\leq a_i\leq 1$ for each $i$.
    For $N\in\N$ and $\mathbf{v}^{(i)}\in\pi_i(\Omega)\vert_N,i=1,\ldots,r-1$,
    denote by
    \begin{equation*}
        Z_N(\mathbf{v}^{(r-1)}):=|(\tau_{1}\vert_N)^{-1}(\mathbf{v}^{(r-1)})|
    \end{equation*}
    the cardinality of $(\tau_{1}\vert_N)^{-1}(\mathbf{v}^{(r-1)})\subset \Omega\vert_N$. More generally, if $Z_N(\mathbf{v}^{(r-i+1)})$ is already defined, let
    \begin{equation*}
        Z_N(\mathbf{v}^{(r-i)}):=\sum_{\mathbf{v}^{(r-i+1)}\in(\tau_{i}\vert_N)^{-1}(\mathbf{v}^{(r-i)})}Z_N(\mathbf{v}^{(r-i+1)})^{a_{i-1}}
    \end{equation*}
    for $i=2,\ldots,r-1$. Finally, let
    \begin{equation*}
        Z_N:=\sum_{\mathbf{v}^{(1)}\in\pi_1(\Omega)\vert_N}Z_N(\mathbf{v}^{(1)})^{a_{r-1}}.
    \end{equation*}
    Then, the weighted topological entropy of $\mathbf{a}$-exponent is given by
    \begin{equation*}
        \htop^{\mathbf{a}}(\{(\pi_i(\Omega),\sigma)\}_{i=1}^r,\{\tau_i\}_{i=1}^{r-1})=\lim_{N\to\infty}\frac{\log Z_N}{N}.
    \end{equation*}
    Notice that the limit exists since $\log Z_N$ is sub-additive in $N$.
\end{lemma}

\begin{proof}
    It is direct to check that
    \begin{equation*}
        \htop^{\mathbf{a}}(\{(\pi_i(\Omega),\sigma)\}_{i=1}^r,\{\tau_i\}_{i=1}^{r-1})\leq\lim_{N\to\infty}\frac{\log Z_N}{N}
    \end{equation*}
    so we verify the opposite inequality. Fix $1\leq i\leq r$. We define the metric on $\pi_i(\Omega)$ by
    \begin{equation*}
        d^{(r-i+1)}((\mathbf{x}_1,\ldots,\mathbf{x}_i),(\mathbf{y}_1,\ldots,\mathbf{y}_i))=2^{-\min\{k\in\N:(x_{k,1},\ldots,x_{k,i})\neq (y_{k,1},\ldots,y_{k,i})\}}.
    \end{equation*}
    For a subset $U_i\subset(\prod_{l=1}^{i}\mathcal{I}_l)^{\N}$, we denote by $U_i\vert_N\subset\prod_{l=1}^{i}\mathcal{I}_i^{N}$ be projection onto the first $N$ coordinates. If $\diam(U_i,d_N^{(r-i+1)})<1$, then $U_i\vert_N$ is a singleton or empty.

\smallskip
    Fix $0<\varepsilon<1$. Let $\{V_{j_r}: j_r=1,\ldots,n_r\}$ be an open cover of $\pi_1(\Omega)$ with $\diam(V_{j_r},d_N^{(r)})<\varepsilon$ for each $j_r$ and
    \begin{equation*}
        \#_r^{\mathbf{a}}(\pi_1(\Omega),N,\varepsilon)=\sum_{j_r=1}^{n_r}\#_{r-1}^{\mathbf{a}}(\tau_{r-1}^{-1}(V_{j_r}),N,\varepsilon)^{a_{r-1}}.
    \end{equation*}
    For each $j_r$, let $\{V_{j_r,j_{r-1}}: j_{r-1}=1,\ldots,n_{r-1}(j_r)\}$ be an open cover of $\tau_{r-1}^{-1}(V_{j_r})$ with $\diam(V_{j_r,j_{r-1}},d_N^{(r-1)})<\varepsilon$ and
    \begin{equation*}
        \#_{r-1}^{\mathbf{a}}(\tau_{r-1}^{-1}(V_{j_r}),N,\varepsilon)=\sum_{j_{r-1}=1}^{n_{r-1}(j_r)}\#_{r-2}^{\mathbf{a}}(\tau_{r-2}^{-1}(V_{j_r,j_{r-1}}),N,\varepsilon)^{a_{r-2}}.
    \end{equation*}
    We repeat the process until we obtain an open cover \\$\tau_{2}^{-1}(V_{j_r,j_{r-1},\cdots, j_3})=\bigcup_{j_2=1}^{n_2(j_3)}V_{j_r,j_{r-1},\cdots,j_2}$ with $\diam(V_{j_r,\cdots,j_2},d_N^{(2)})<\varepsilon$ and
    \begin{equation*}
        \#_{2}^{\mathbf{a}}(\tau_{2}^{-1}(V_{j_r,\cdots,j_3}),N,\varepsilon)=\sum_{j_2=1}^{n_2(j_3)}\#_{1}^{\mathbf{a}}(\tau_1^{-1}(V_{j_r,\cdots,j_2}),N,\varepsilon)^{a_1}.
    \end{equation*}
    Furthermore, for each $(j_r,\cdots, j_2)$ let $\tau_1^{-1}(V_{j_r,\cdots, j_2})=\bigcup_{j_1=1}^{n_1(j_2)}V_{j_r,\cdots,j_1}$ with \\ $\diam(V_{j_r,\cdots, j_1},d_N^{(1)})<\varepsilon$. Any $V_{j_r,\cdots, j_1}$ is not empty, and therefore $V_{j_r,\cdots, j_1}\vert_N$ is a singleton. We have
    \begin{equation*}
        \begin{split}
            \#_r^{\mathbf{a}}(\pi_1(\Omega),N,\varepsilon)
            &=\sum_{j_r=1}^{n_r}\left( \sum_{j_{r-1}=1}^{n_{r-1}(j_r)} \left( \cdots\left( \sum_{j_2=1}^{n_2(j_3)}n_1(j_2)^{a_1} \right)^{a_2}\cdots\right)^{a_{r-2}}\right)^{a_{r-1}}.\\
        \end{split}
    \end{equation*}
    Fix $\mathbf{v}^{(1)}\in\pi_1(\Omega)\vert_N,\mathbf{v}^{(2)}\in(\tau_{r-1}\vert_N)^{-1}(\mathbf{v}^{(1)}),\ldots$, and $\mathbf{v}^{(r-1)}\in(\tau_{2}\vert_N)^{-1}(\mathbf{v}^{(r-2)})$. Note that
    \begin{equation*}
    \begin{split}
        (\tau_{1}\vert_N)^{-1}(\mathbf{v}^{(r-1)})&=\bigcup_{j_r:V_{j_r}\vert_N=\{\mathbf{v}^{(1)}\}} \cup\cdots\cup
        \bigcup_{j_2:V_{j_r,\cdots, j_2}\vert_N=\{\mathbf{v}^{(r-1)}\}} (\tau_1^{-1}(V_{j_r,\cdots, j_2}))\vert_N \\
        &=\bigcup_{j_r:V_{j_r}\vert_N=\{\mathbf{v}^{(1)}\}} \cup\cdots\cup
        \bigcup_{j_2:V_{j_r,\cdots, j_2}\vert_N=\{\mathbf{v}^{(r-1)}\}} \bigcup_{j_1=1}^{n_1(j_2)} (V_{j_r,\cdots, j_1}\vert_N).
    \end{split}
    \end{equation*}
    Since every $V_{j_r,\cdots, j_1}\vert_N$ is a singleton,
    \begin{equation*}
        Z_N(\mathbf{v}^{(r-1)})=|(\tau_{1}\vert_N)^{-1}(\mathbf{v}^{(r-1)})|\leq \sum_{j_r:V_{j_r}\vert_N=\{\mathbf{v}^{(1)}\}}\cdots\sum_{j_2:V_{j_r,\cdots, j_2}\vert_N=\{\mathbf{v}^{(r-1)}\}} n_1(j_2)
    \end{equation*}
    and thus
    \begin{equation*}
    {\tiny \begin{split}
        &Z_N(\mathbf{v}^{(r-2)})\\[2mm]
        &=\sum_{\mathbf{v}^{(r-1)}\in(\tau_{2}\vert_N)^{-1}(\mathbf{v}^{(r-2)})}  |(\tau_{1}\vert_N)^{-1}(\mathbf{v}^{(r-1)})|^{a_1} \\[2mm]
        &\leq \hspace*{-2pt}\sum_{j_r:V_{j_r}\vert_N=\{\mathbf{v}^{(1)}\}}\hspace*{-2pt}\cdots\hspace*{-2pt}\sum_{j_3:V_{j_r,\cdots, j_3}\vert_N=\{\mathbf{v}^{(r-2)}\}} \sum_{\mathbf{v}^{(r-1)}\in(\tau_{2}\vert_N)^{-1}(\mathbf{v}^{(r-2)})} \sum_{j_2:V_{j_r,\cdots, j_2}\vert_N=\{\mathbf{v}^{(r-1)}\}} n_1(j_2)^{a_1},
    \end{split}}
    \end{equation*}
    where we have used an elementary inequality: $(x+y)^a\leq x^a+y^a$ for all $x,y\geq 0$ and $0\leq a\leq 1$.
    Furthermore, one has
    \begin{equation*}
        \begin{split}
            &Z_N(\mathbf{v}^{(r-3)})\\
            &=\sum_{\mathbf{v}^{(r-2)}\in(\tau_{3}\vert_N)^{-1}(\mathbf{v}^{(r-3)})} {Z_N(\mathbf{v}^{(r-2)})}^{a_2}\\
            &\leq \sum_{j_r:V_{j_r}\vert_N=\{\mathbf{v}^{(1)}\}}\cdots \sum_{\mathbf{v}^{(r-2)}\in(\tau_{3}\vert_N)^{-1}(\mathbf{v}^{(r-3)})}\\
            &\quad  \left(  \sum_{j_3:V_{j_r,\cdots, j_3}\vert_N=\{\mathbf{v}^{(r-2)}\}} \sum_{\mathbf{v}^{(r-1)}\in(\tau_{2}\vert_N)^{-1}(\mathbf{v}^{(r-2)})} \sum_{j_2:V_{j_r,\cdots, j_2}\vert_N=\{\mathbf{v}^{(r-1)}\}} n_1(j_2)^{a_1} \right)^{a_2}.
        \end{split}
    \end{equation*}
    Repeating the above calculation, we finally obtain that
    \begin{equation*}
        \begin{split}
            Z_N&\leq \sum_{\mathbf{v}^{(1)}\in\pi_1(\Omega)\vert_N} \sum_{j_r:V_{j_r}\vert_N=\{\mathbf{v}^{(1)}\}}
            \left( \sum_{\mathbf{v}^{(2)}\in(\tau_{r-1}\vert_N)^{-1}(\mathbf{v}^{(1)})} \sum_{j_{r-1}:V_{j_r,j_{r-1}}\vert_N=\{\mathbf{v}^{(2)}\}} \right.\\
            &\quad \left.\left( \cdots \left( \sum_{\mathbf{v}^{(r-1)}\in(\tau_{2}\vert_N)^{-1}(\mathbf{v}^{(r-2)})}
            \sum_{j_2:V_{j_r,\cdots, j_2}\vert_N=\{\mathbf{v}^{(r-1)}\}} n_1(j_2)^{a_1} \right)^{a_2} \cdots \right)^{a_{r-2}} \right)^{a_{r-1}}\\
            &=\#_r^{\mathbf{a}}(\pi_1(\Omega),N,\varepsilon).
        \end{split}
    \end{equation*}
    Therefore
    \begin{equation*}
        \lim_{N\to\infty}\frac{\log Z_N}{N}\leq  \lim_{N\to\infty}\frac{\log\#_r^{\mathbf{a}}(\pi_1(\Omega),N,\varepsilon)}{N} \leq \htop^{\mathbf{a}}(\{(\pi_i(\Omega),\sigma)\}_{i=1}^r,\{\tau_i\}_{i=1}^{r-1}).
    \end{equation*}
\end{proof}

Remark that in most situations, it is difficult to give an explicit value of the weighted topological entropy using the formula obtained by Lemma \ref{Formula for the weighted topological entropy} above. However, below we discuss a special case for which the weighted topological entropy is expressed in terms of combinatorial standard topological entropy, which further yields the coincidence of the metric mean dimension and the mean Hausdorff dimension for certain sponge systems.

\begin{lemma}\label{coincidence}
    Let $(X_{\Omega},\sigma,d)$ be a sponge system defined in Section \ref{defn:sponge system}. Assume that for every $1\leq i\leq r-1$ and $N\in\N$, there exists a sequence $\left(P_N^{(r-i+1)}\right)_{N\in\N}$ satisfying the following uniform convergence\footnote{This condition was proposed by Alibabaei (\cite[Example 2.4]{Ali24}) and $\Omega$ is said to have \textbf{uniformly growing word complexity}.}:
    \begin{equation}\label{uniformly growing condition}
        \lim_{N\to\infty}\left( \frac{|(\tau_i\vert_N)^{-1}(\mathbf{v}^{(r-i)})|}{P_N^{(r-i+1)}} \right)^{\frac{1}{N}}=1
    \end{equation}
    for any $\mathbf{v}^{(r-i)}\in\pi_i(\Omega)\vert_N$, where $\tau_i:(\pi_{r-i+1}(\Omega),\sigma)\to(\pi_{r-i}(\Omega),\sigma), 1\leq i\leq r-1$ are factor maps. Then, $\mhdim(X_{\Omega},\sigma,d)=\mmdim(X_{\Omega},\sigma,d)$.
\end{lemma}

\begin{proof}[Proof of Lemma \ref{coincidence} assuming Theorem \ref{mdim of sponge system} from Section \ref{defn:sponge system} above]
    {\small Fix $N\in\N$. Let} $P_N^{(1)}=|\pi_1(\Omega)\vert_N|$. Set $a_i=\log m_{r-i}/\log m_{r-i+1}$ for each $1\leq i\leq r-1$. Let $\mathbf{w_a}=(w_1,\ldots,w_r)$ be defined by \eqref{wa}. Then,
    \begin{equation*}
        \mathbf{w_a}=\left( \frac{\log m_1}{\log m_r}, \frac{\log m_1}{\log m_{r-1}}-\frac{\log m_1}{\log m_r},\ldots, \frac{\log m_1}{\log m_2}-\frac{\log m_1}{\log m_3}, 1-\frac{\log m_1}{\log m_2} \right).
    \end{equation*}
    Combining Lemma \ref{Formula for the weighted topological entropy} and \eqref{uniformly growing condition}, we have
    \begin{equation*}
       {\tiny \begin{split}
            &\htop^{\mathbf{a}}(\{(\pi_i(\Omega),\sigma)\}_{i=1}^r,\{\tau_i\}_{i=1}^{r-1})\\[2mm]
            &=\lim_{N\to\infty}\frac{\log Z_N}{N}\\[2mm]
            &=\lim_{N\to\infty}\frac{1}{N} \log \left( P_N^{(1)} {P_N^{(2)}}^{a_{r-1}}\cdots{P_N^{(r-1)}}^{a_2a_3\ldots a_{r-1}} {P_N^{(r)}}^{a_1a_2\ldots a_{r-1}}\right) \\[2mm]
            &=\lim_{N\to\infty}\hspace*{-2pt}\frac{1}{N}\left(\log\left(P_N^{(1)} \cdots P_N^{(r)}\right)^{a_1\ldots a_{r-1}}
            \hspace*{-6pt}+\sum\limits_{i=2}^{r-1} \log\left( P_N^{(1)} \cdots P_N^{(i)} \right)^{(1-a_{r-i})a_{r-i+1}\cdots a_{r-1}}   \hspace*{-8pt} + \log \left(P_N^{(1)}\right)^{1-a_{r-1}} \right)\\[2mm]
            &=\lim_{N\to\infty}\frac{1}{N}\left(\log\left(P_N^{(1)} \cdots P_N^{(r)}\right)^{w_1}
            +\sum\limits_{i=1}^{r-1} \log\left( P_N^{(1)} \cdots P_N^{(i)} \right)^{w_{r-i+1}}     \right)\\[2mm]
            &=\lim_{N\to\infty}\frac{1}{N} \left( w_1\log|\Omega\vert_N| + \sum\limits_{i=1}^{r-1}w_{r-i+1} \log |\pi_i(\Omega)\vert_N| \right)\\[2mm]
            &=w_1\htop(\Omega,\sigma)+\sum\limits_{i=1}^{r-1}w_{r-i+1}\htop(\pi_i(\Omega),\sigma).
        \end{split}}
    \end{equation*}
    By Theorem \ref{mdim of sponge system}, we further have
    \begin{equation*}
        \begin{split}
            \mhdim(X_{\Omega},\sigma,d)
            &=\frac{\htop^{\mathbf{a}}(\{(\pi_i(\Omega),\sigma)\}_{i=1}^r,\{\tau_i\}_{i=1}^{r-1})}{\log m_1} \\[2mm]
            &=\frac{\htop(\Omega,\sigma)}{\log m_r}+\sum\limits_{i=1}^{r-1}\left(\frac{1}{\log m_{i}}-\frac{1}{\log m_{i+1}}\right)\htop(\pi_i(\Omega),\sigma)\\[2mm]
            &=\mmdim(X_{\Omega},\sigma,d).
        \end{split}
    \end{equation*}
\end{proof}

\subsection{Preparations}
In this section, we collect some basic tools that will be used to calculate the mean Hausdorff dimension of the sponge systems.

First, we recall a fundamental result of geometric measure theory by Tsukamoto \cite[Lemma 5.7]{Tsu22}.
\begin{lemma}\label{Hausdorff dim estimate}
    Let $c,\varepsilon,s$ be positive constants. Let $(X,d)$ be a compact metrizable space with a Borel probability measure $\mu$. Suppose that
    \begin{enumerate}[$(1)$]
        \item $6\varepsilon^c<1$.
        \item For any $x\in X$ there exists a Borel subset $A\subset X$ containing $x$ and satisfying
        \begin{equation*}
            0<\diam A<\frac{\varepsilon}{6}, \mu(A)\geq (\diam A)^s.
        \end{equation*}
    \end{enumerate}
    Then, $\hdim(X,d,\varepsilon)\leq (1+c)s$.
\end{lemma}

The next combinatorial lemma is obtained by Kenyon and Peres in \cite[Lemma 4.1]{KP96}.
\begin{lemma}\label{KP96,Lemma4.1}
    Fix $r\in\N$. Let $u_i:\N\to\R \ (i=1,\ldots,r)$ satisfy
    \begin{equation*}
        \sup_{n\in\N}|u_i(n+1)-u_i(n)|<\infty.
    \end{equation*}
    Then, for any positive numbers $c_1,\ldots,c_{r}$ and $L_1,\ldots,L_{r}$,
    \begin{equation*}
        \limsup_{t\to\infty}\frac{1}{t}\sum_{i=1}^r\left\{c_iu_i\left(\left\lfloor \frac{t}{L_i} \right\rfloor\right)-u_i\left(\left\lfloor \frac{c_it}{L_i} \right\rfloor\right)\right\}\geq 0.
    \end{equation*}
\end{lemma}

\begin{lemma}\cite[Lemma 5.9 (1)]{Tsu22}\label{Lemma5.9(1)}
    Let $\eta_1,\eta_2,\eta_3,\ldots$ be independent and identically distributed (i.i.d.) random variables with $\eta_1\in L^2$. Then, for any $\delta>0$ and $m\in\N$,
    \begin{equation*}
        \mathbb{P}\left(\sup\limits_{n\geq m} \left| \frac{\sum_{k=1}^{n}\eta_k}{n}-\mathbb{E}\eta_1 \right|>\delta \right)\leq\frac{4\mathbb{V}(\eta_1)}{\delta^2 m}.
    \end{equation*}
\end{lemma}

The following lemma is a slight variant of  \cite[Lemma 5.9 (3)]{Tsu22}, but we include the details here for completeness.
\begin{lemma}\label{Lemma5.9}
    Let $r\in\N$ and $0<c_i,\widehat{c}_i\leq 1$ for all $1\leq i\leq r$. For any positive constants $\delta$ and $C$, there exists $M_0\in\N$ depending on $c_i,\widehat{c}_i,i=1,\ldots,r,C$ and $\delta$ for which the following statement holds: If for each fixed $1\leq i\leq r$, $\xi_{k,i},k\in\N$ are i.i.d. random variables with $\xi_{1,i}\in L^2, \max\{|\mathbb{E}\xi_{1,i}|,\mathbb{V}\xi_{1,i}\}\leq C$, then
    \begin{equation*}
        \mathbb{P}\left( \sup_{M\geq M_0}\sum_{i=1}^{r}\left|  \frac{\sum_{k=1}^{\lfloor c_iM \rfloor}\xi_{k,i}}{c_iM}
        -\frac{\sum_{k=1}^{\lfloor \widehat{c}_iM \rfloor}\xi_{k,i}}{\widehat{c}_iM}  \right|\leq\delta \right)\geq\frac{1}{2}.
    \end{equation*}
\end{lemma}
\begin{proof}
    Fix $1\leq i\leq r$. Take $M'_0\in\N$ with $\frac{4C}{(\delta/4r)^2 M'_0}\leq 1/4r$.
    By Lemma \ref{Lemma5.9(1)},
    \begin{equation*}
        \mathbb{P}\left( \sup_{M\geq M'_0} \left| \frac{\sum_{k=1}^M\xi_{k,i}}{M}-\mathbb{E}\xi_{1,i} \right|\geq\frac{\delta}{4r} \right)\leq\frac{1}{4r}.
    \end{equation*}
    Choose $M_0\in\N$ satisfying
    \begin{equation*}
        M_0\geq\max_{1\leq i\leq r}\max\left\{ \frac{M'_0}{c_i},\frac{M'_0}{\widehat{c}_i},\frac{4rC}{c_i\delta},\frac{4rC}{\widehat{c}_i\delta} \right\}.
    \end{equation*}
    For all $M\geq M_0$, we have $\lfloor c_iM \rfloor\geq M'_0$, and therefore
    \begin{equation*}
        \mathbb{P}\left( \sup_{M\geq M_0} \left| \frac{\sum_{k=1}^{\lfloor c_iM \rfloor}\xi_{k,i}}{\lfloor c_iM \rfloor}-\mathbb{E}\xi_{1,i} \right|\geq\frac{\delta}{4r} \right)\leq\frac{1}{4r}.
    \end{equation*}
    Notice that
    \begin{equation*}
        \begin{split}
            \left\{ \sup_{M\geq M_0}\sum_{i=1}^{r}\left| \frac{\sum_{k=1}^{\lfloor c_iM \rfloor}\xi_{k,i}}{\lfloor c_iM \rfloor}
        -\mathbb{E}\xi_{1,i} \right|\geq\frac{\delta}{4} \right\}
        \subset\bigcup_{i=1}^{r}\left\{ \sup_{M\geq M_0} \left| \frac{\sum_{k=1}^{\lfloor c_iM \rfloor}\xi_{k,i}}{\lfloor c_iM \rfloor}-\mathbb{E}\xi_{1,i} \right|\geq\frac{\delta}{4r} \right\}.
        \end{split}
    \end{equation*}
    Then,
    \begin{equation*}
        \mathbb{P}\left( \sup_{M\geq M_0}\sum_{i=1}^{r}\left| \frac{\sum_{k=1}^{\lfloor c_iM \rfloor}\xi_{k,i}}{\lfloor c_iM \rfloor}
        -\mathbb{E}\xi_{1,i} \right|\geq\frac{\delta}{4} \right)\leq\frac{1}{4},
    \end{equation*}
    and thus
    \begin{equation*}
        \mathbb{P}\left( \sup_{M\geq M_0}\sum_{i=1}^{r}\left|\frac{\sum_{k=1}^{\lfloor c_iM \rfloor}\xi_{k,i}}{\lfloor c_iM \rfloor}
        -\mathbb{E}\xi_{1,i} \right|\leq\frac{\delta}{4} \right)\geq\frac{3}{4}.
    \end{equation*}
    Straightforward calculation shows that
    \begin{equation*}
        \left| \frac{\sum_{k=1}^{\lfloor c_iM \rfloor}\xi_{k,i}}{c_iM}-\mathbb{E}\xi_{1,i} \right| \leq
        \left| \frac{\sum_{k=1}^{\lfloor c_iM \rfloor}\xi_{k,i}}{\lfloor c_iM \rfloor}-\mathbb{E}\xi_{1,i} \right| + \frac{C}{c_i M}
    \end{equation*}
    for each $i$. Indeed,
    \begin{equation*}
        \begin{split}
            \left| \frac{\sum_{k=1}^{\lfloor c_iM \rfloor}\xi_{k,i}}{c_iM}-\mathbb{E}\xi_{1,i} \right|
            &=\left| \frac{\lfloor c_iM \rfloor}{c_iM}\left( \frac{\sum_{k=1}^{\lfloor c_iM \rfloor}\xi_{k,i}}{\lfloor c_iM \rfloor}-\mathbb{E}\xi_{1,i} \right) + \left( \frac{\lfloor c_iM \rfloor}{c_iM}-1\right)\mathbb{E}\xi_{1,i} \right|\\
            &\leq \left| \frac{\sum_{k=1}^{\lfloor c_iM \rfloor}\xi_{k,i}}{\lfloor c_iM \rfloor}-\mathbb{E}\xi_{1,i} \right| +\frac{c_iM-\lfloor c_iM \rfloor}{c_iM}|\mathbb{E}\xi_{1,i}|
        \end{split}
    \end{equation*}
    and $|\mathbb{E}\xi_{1,i}|\leq C$.
    Since $M\geq M_0$ implies that $rC/c_iM\leq \delta/4$, we have
    \begin{equation*}
        \sup_{M\geq M_0} \sum_{i=1}^{r}\left| \frac{\sum_{k=1}^{\lfloor c_iM \rfloor}\xi_{k,i}}{c_iM}-\mathbb{E}\xi_{1,i} \right| \leq
        \sup_{M\geq M_0} \sum_{i=1}^{r}\left| \frac{\sum_{k=1}^{\lfloor c_iM \rfloor}\xi_{k,i}}{\lfloor c_iM \rfloor}-\mathbb{E}\xi_{1,i} \right| +\frac{\delta}{4}.
    \end{equation*}
    Furthermore, we have
    \begin{equation*}
        \mathbb{P}\left( \sup_{M\geq M_0} \sum_{i=1}^{r}\left| \frac{\sum_{k=1}^{\lfloor c_iM \rfloor}\xi_{k,i}}{c_iM}-\mathbb{E}\xi_{1,i} \right|\leq\frac{\delta}{2} \right)\geq\frac{3}{4}.
    \end{equation*}
    Similarly, we obtain that
    \begin{equation*}
        \mathbb{P}\left( \sup_{M\geq M_0} \sum_{i=1}^{r} \left| \frac{\sum_{k=1}^{\lfloor \widehat{c}_iM \rfloor}\xi_{k,i}}{\widehat{c}_iM}-\mathbb{E}\xi_{1,i} \right|\leq\frac{\delta}{2} \right)\geq\frac{3}{4}.
    \end{equation*}
    Observe that
    \begin{equation*}
        \begin{split}
           & \left\{ \sup_{M\geq M_0}\sum_{i=1}^{r}\left|  \frac{\sum_{k=1}^{\lfloor c_iM \rfloor}\xi_{k,i}}{c_iM}
        -\frac{\sum_{k=1}^{\lfloor \widehat{c}_iM \rfloor}\xi_{k,i}}{\widehat{c}_iM}  \right|\leq\delta \right\} \\
        &\supset\left\{ \sup_{M\geq M_0}\sum_{i=1}^{r} \left| \frac{\sum_{k=1}^{\lfloor c_iM \rfloor}\xi_{k,i}}{c_iM}-\mathbb{E}\xi_{1,i} \right|\leq\frac{\delta}{2} \right\} \\
        &\bigcap
         \left\{ \sup_{M\geq M_0}\sum_{i=1}^{r} \left| \frac{\sum_{k=1}^{\lfloor \widehat{c}_iM \rfloor}\xi_{k,i}}{\widehat{c}_iM}-\mathbb{E}\xi_{1,i} \right|\leq\frac{\delta}{2} \right\}.
        \end{split}
    \end{equation*}
    Finally, we obtain that
    \begin{equation*}
        \mathbb{P}\left( \sup_{M\geq M_0}\sum_{i=1}^{r}\left|  \frac{\sum_{k=1}^{\lfloor c_iM \rfloor}\xi_{k,i}}{c_iM}
        -\frac{\sum_{k=1}^{\lfloor \widehat{c}_iM \rfloor}\xi_{k,i}}{\widehat{c}_iM}  \right|\leq\delta \right)\geq\frac{1}{2}.
    \end{equation*}
\end{proof}

\subsection{Calculation of mean Hausdorff dimension}
In this section, we prove \eqref{mhdim of sponge system} in Theorem \ref{mdim of sponge system}. More precisely, in subsections \ref{mhdim:upper}-\ref{mhdim:lower} we prove the upper and lower bounds of \eqref{mhdim of sponge system}, respectively. The proof adopts the ideas in \cite[Section 5.5]{Tsu22} and \cite[Section 4]{KP96}.

Fix $N\in\N$. Recall that $Z_N$ is inductively defined in Lemma \ref{Formula for the weighted topological entropy}. Notice that
\begin{equation*}
    {\SMALL \begin{split}
        1&=\sum_{\mathbf{v}^{(1)}\in\pi_1(\Omega)\vert_N}\frac{Z_N(\mathbf{v}^{(1)})\cdot Z_N(\mathbf{v}^{(1)})^{a_{r-1}-1}}{Z_N}\\[2mm]
        &=\sum_{\mathbf{v}^{(1)}\in\pi_1(\Omega)\vert_N}\sum_{\mathbf{v}^{(2)}\in(\tau_{r-1}\vert_N)^{-1}(\mathbf{v}^{(1)})}\frac{Z_N(\mathbf{v}^{(2)})\cdot Z_N(\mathbf{v}^{(2)})^{a_{r-2}-1} \cdot Z_N(\mathbf{v}^{(1)})^{a_{r-1}-1}}{Z_N}\\[2mm]
        &=\cdots=\sum_{\mathbf{v}^{(1)}\in\pi_1(\Omega)\vert_N}\cdots\sum_{\mathbf{v}^{(r-1)}\in(\tau_{2}\vert_N)^{-1}(\mathbf{v}^{(r-2)})}\frac{Z_N(\mathbf{v}^{(r-1)})\cdot Z_N(\mathbf{v}^{(r-1)})^{a_1-1}\cdots Z_N(\mathbf{v}^{(1)})^{a_{r-1}-1}}{Z_N}\\[2mm]
        &=\sum_{(\mathbf{u}^{(1)},\mathbf{u}^{(2)},\ldots,\mathbf{u}^{(r)})\in\Omega\vert_N}
        \frac{Z_N(\mathbf{u}^{(1)},\ldots,\mathbf{u}^{(r-1)})^{a_1-1}\cdot Z_N(\mathbf{u}^{(1)},\ldots,\mathbf{u}^{(r-2)})^{a_{2}-1} \cdots Z_N(\mathbf{u}^{(1)})^{a_{r-1}-1}}{Z_N}.
    \end{split}}
\end{equation*}

\smallskip
For $(\mathbf{u}^{(1)},\mathbf{u}^{(2)},\ldots,\mathbf{u}^{(r)})\in\Omega\vert_N$, we set
\begin{equation*}
  \begin{aligned}  &f_N(\mathbf{u}^{(1)},\mathbf{u}^{(2)},\ldots,\mathbf{u}^{(r)})\\
    &=\frac{1}{Z_N}Z_N(\mathbf{u}^{(1)},\ldots,\mathbf{u}^{(r-1)})^{a_1-1}\cdot Z_N(\mathbf{u}^{(1)},\ldots,\mathbf{u}^{(r-2)})^{a_{2}-1} \cdots Z_N(\mathbf{u}^{(1)})^{a_{r-1}-1}.
\end{aligned}\end{equation*}
Then, we have
\begin{equation*}
    \sum_{(\mathbf{u}^{(1)},\mathbf{u}^{(2)},\ldots,\mathbf{u}^{(r)})\in\Omega\vert_N} f_N(\mathbf{u}^{(1)},\mathbf{u}^{(2)},\ldots,\mathbf{u}^{(r)})=1
\end{equation*}
and therefore $f_N(\mathbf{u}^{(1)},\mathbf{u}^{(2)},\ldots,\mathbf{u}^{(r)})$ is indeed a probability measure on $\Omega\vert_N$ and independent of $\mathbf{u}^{(r)}$. Let $\mu_N=(f_N)^{\otimes\N}$ be the countable product of the measure $f_N$. Then, $\mu_N$ is a probability measure defined on $(\Omega\vert_N)^{\N}$.

\smallskip
Fix $M\in\N$. Let $(\mathbf{x}_1,\ldots,\mathbf{x}_r)=((\mathbf{x}_{k,1})_{k\in\N},\ldots,(\mathbf{x}_{k,r})_{k\in\N})\in (\Omega\vert_N)^{\N}$ with $\mathbf{x}_{k,l}\in\mathcal{I}_l^N$ and $(\mathbf{x}_{k,1},\ldots,\mathbf{x}_{k,r})\in\Omega\vert_N$ for each $k\in\N$.
We set
\begin{equation*}
    \begin{split}
        P_{N,M}(\mathbf{x}_1,\ldots,\mathbf{x}_r)&=\left\{(\mathbf{y}_1,\ldots,\mathbf{y}_r)=((\mathbf{y}_{k,1})_{k\in\N},\ldots,(\mathbf{y}_{k,r})_{k\in\N})\in (\Omega\vert_N)^{\N}:  \right.\\[2mm]
      &\left.\mathbf{y}_{k,i}=\mathbf{x}_{k,i} \text{  for all } 1\leq k\leq L_i(M) \text{ and } 1\leq i \leq r\right\}.
    \end{split}
\end{equation*}
Then, it is clear the set $Q_{N,M}(\mathbf{x}_1,\ldots,\mathbf{x}_r)$ defined by \eqref{approximate cube} is the image of $P_{N,M}(\mathbf{x}_1,\ldots,\\ \mathbf{x}_r)$ under the map
\begin{equation*}
    \begin{split}
        (\Omega\vert_N)^{\N}\to X_{\Omega}\vert_N,  (\mathbf{x}_1,\ldots,\mathbf{x}_r)\mapsto \left(\sum\limits_{k=1}^{\infty} \frac{\mathbf{x}_{k,1}}{m_1^k},\ldots,\sum\limits_{k=1}^{\infty}\frac{\mathbf{x}_{k,r}}{m_r^k}\right).
    \end{split}
\end{equation*}
Moreover,
\begin{equation*}
       {\Small  \begin{split}
           (\Omega\vert_N)^{\N}=&\bigcup\left\{ P_{N,M}(\mathbf{x}_1,\ldots,\mathbf{x}_r): (\mathbf{x}_{k,1},\ldots,\mathbf{x}_{k,r})\in\Omega\vert_N  \text{ for } 1\leq k\leq L_r(M)\right.\\[2mm]
           &\left. (\mathbf{x}_{k,1},\ldots,\mathbf{x}_{k,i})\in \pi_i(\Omega)\vert_N \text{ for each } L_{i+1}(M)+1\leq k\leq L_i(M) \text{ and } 1\leq i\leq r-1 \right\}.
        \end{split}}
\end{equation*}
We now calculate the $\mu_N$-measure of $P_{N,M}(\mathbf{x}_1,\ldots,\mathbf{x}_r)$.
For $(\mathbf{u}^{(1)},\mathbf{u}^{(2)},\ldots,\mathbf{u}^{(r)})\in\Omega\vert_N$, set
\begin{equation*}
    f_N^{(1)}(\mathbf{u}^{(1)})=\frac{Z_N(\mathbf{u}^{(1)})^{a_{r-1}}}{Z_N},
\end{equation*}
\begin{equation*}
    f_N^{(i)}(\mathbf{u}^{(i)}\vert\mathbf{u}^{(1)},\ldots,\mathbf{u}^{(i-1)})=\frac{Z_N(\mathbf{u}^{(1)},\ldots\mathbf{u}^{(i)})^{a_{r-i}}}{Z_N(\mathbf{u}^{(1)},\ldots,\mathbf{u}^{(i-1)})}
\end{equation*}
for $i=2,\ldots,r-1$ and
\begin{equation*}
    f_N^{(r)}(\mathbf{u}^{(r)}\vert\mathbf{u}^{(1)},\ldots,\mathbf{u}^{(r-1)})=\frac{1}{Z_N(\mathbf{u}^{(1)},\ldots,\mathbf{u}^{(r-1)})}.
\end{equation*}
Then, we may rewrite
\begin{equation*}
    f_N(\mathbf{u}^{(1)},\mathbf{u}^{(2)},\ldots,\mathbf{u}^{(r)})=f_N^{(1)}(\mathbf{u}^{(1)})\cdot\prod_{i=2}^{r}f_N^{(i)}(\mathbf{u}^{(i)}\vert\mathbf{u}^{(1)},\ldots,\mathbf{u}^{(i-1)}).
\end{equation*}
Note that the value of $f_N^{(i)}(\mathbf{u}^{(i)}\vert\mathbf{u}^{(1)},\ldots,\mathbf{u}^{(i-1)})$ only depends on $\mathbf{u}^{(1)},\mathbf{u}^{(2)},\ldots,\mathbf{u}^{(i)}$. Indeed, the number $f_N^{(i)}(\mathbf{u}^{(i)}\vert\mathbf{u}^{(1)},\ldots,\mathbf{u}^{(i-1)})$ can be interpreted as the conditional probability that the $i$th digit of a randomly chosen member of $\Omega\vert_N$ equals $\mathbf{u}^{(i)}$, given that the first $i-1$ coordinates did. Then, we have
\begin{equation*}
\begin{split}
    \mu_N(P_{N,M}(\mathbf{x}_1,\ldots,\mathbf{x}_r))&=
    \prod_{k=1}^{L_1(M)}f_N^{(1)}(\mathbf{x}_{k,1})\cdot
    \prod_{i=2}^{r}\prod_{k=1}^{L_i(M)} f_N^{(i)}(\mathbf{x}_{k,i}\vert\mathbf{x}_{k,1},\ldots,\mathbf{x}_{k,i-1})\\[2mm]
    &=Z_N^{-M}\cdot \prod_{i=1}^{r-1}\prod_{k=1}^{L_i(M)} Z_N(\mathbf{x}_{k,1},\ldots,\mathbf{x}_{k,i})^{a_{r-i}-1}\\[2mm]
    &\quad \cdot \prod_{i=1}^{r-1}\prod_{k=L_{i+1}(M)+1}^{L_i(M)} Z_N(\mathbf{x}_{k,1},\ldots,\mathbf{x}_{k,i}).
\end{split}
\end{equation*}
Taking a logarithm, we further have
\begin{equation*}
    {\small \begin{split}
        &\log \mu_N(P_{N,M}(\mathbf{x}_1,\ldots,\mathbf{x}_r))\\
        &=-M\log Z_N\\
        &\quad +\sum_{i=1}^{r-1}\left( (a_{r-i}-1)\sum_{k=1}^{L_i(M)}\log Z_N(\mathbf{x}_{k,1},\ldots,\mathbf{x}_{k,i}) + \sum_{k=L_{i+1}(M)+1}^{L_i(M)} \log Z_N(\mathbf{x}_{k,1},\ldots,\mathbf{x}_{k,i}) \right)   \\
        &=-M\log Z_N +\sum_{i=1}^{r-1}\left(a_{r-i}\sum_{k=1}^{L_i(M)}\log Z_N(\mathbf{x}_{k,1},\ldots,\mathbf{x}_{k,i})
        -\sum_{k=1}^{L_{i+1}(M)} \log Z_N(\mathbf{x}_{k,1},\ldots,\mathbf{x}_{k,i})\right).
    \end{split}}
\end{equation*}
Set
\begin{equation}\label{ui}
    u_i(L):=\sum_{k=1}^{L}\frac{\log Z_N(\mathbf{x}_{k,1},\ldots,\mathbf{x}_{k,i})}{N}
\end{equation}
for $i=1,\ldots,r-1$. Notice that this quantity depends not only on $L$ but also $N$ and $(\mathbf{x}_{k,1},\ldots,\mathbf{x}_{k,i})\in \pi_i(\Omega)\vert_N$, but we suppress the dependence on $N$ and $(\mathbf{x}_{k,1},\ldots,\mathbf{x}_{k,i})$ for simplicity. Then,
\begin{equation}\label{(5.8)}
    \frac{1}{NM}\log \mu_N(P_{N,M}(\mathbf{x}_1,\ldots,\mathbf{x}_r))
       \hspace*{-2pt} =\hspace*{-2pt}-\frac{\log Z_N}{N}+\frac{1}{M}\sum_{i=1}^{r-1}(a_{r-i}u_i(L_i(M))-u_{i}(L_{i+1}(M))).
\end{equation}

\begin{lemma}\label{Lemma5.11}
    For each $\delta>0$ and $K\in\N$, there exists $L=L(\delta,K)\geq K$ such that for any $N\in\N$ and any $(\mathbf{x}_1,\ldots,\mathbf{x}_r)\in (\Omega\vert_N)^{\N}$, there exists $M\in[K,L]$ satisfying
    \begin{equation*}
        \frac{1}{NM}\log \mu_N(P_{N,M}(\mathbf{x}_1,\ldots,\mathbf{x}_r))
        \geq-\frac{\log Z_N}{N}-\delta.
    \end{equation*}
\end{lemma}

\begin{proof}
Recall that $a_i=\log m_{r-i}/\log m_{r-i+1}$ and $L_i(M)=\lfloor M\log m_1/\log m_i \rfloor$. Taking $t=\log m_1^M$ and $L_i=\log m_i$ gives $\lfloor t/L_i \rfloor=L_i(M)$ and $\lfloor a_{r-i}t/L_i \rfloor=L_{i+1}(M)$.
Let $u_i:\N\to\R, i=1,\ldots,r-1$ be defined by \eqref{ui}. We further define $u_i(0)=0$. Then,
 \begin{equation*}
     \begin{split}
         \sup_{n\in\N_0}|u_i(n+1)-u_i(n)|
         \leq \sum_{j=i+1}^r\log m_j\leq \sum_{j=2}^r\log m_j
     \end{split}
 \end{equation*}
for each $i$. Indeed, the recursive definition of $Z_N$ gives the uniform estimate
\begin{equation}\label{eq:ZN uniform estimate}
    1\leq Z_N(\mathbf{v}^{(i)})\leq\prod_{j=i+1}^r m_j^N
\end{equation}
 for all $\mathbf{v}^{(i)}\in\pi_i(\Omega)\vert_N$.
 Let $\mathcal{U}$ be the set of all tuples $(v_1,\ldots,v_{r-1})$ such that $v_i:\N_0\to\R$ satisfies $v_i(0)=0$ and $\sup_{n\in\N_0}|v_i(n+1)-v_i(n)|\leq\sum_{j=2}^r\log m_j$. Then, $\mathcal{U}$ is compact with the topology of pointwise convergence. For $(v_1,\ldots,v_{r-1})\in\mathcal{U}$ and $M\in\N$, define
\begin{equation*}
    \Phi_M(v_1,\ldots,v_{r-1}):=\frac{1}{M}\sum_{i=1}^{r-1}(a_{r-i}v_i(L_i(M))-v_{i}(L_{i+1}(M))).
\end{equation*}
By Lemma \ref{KP96,Lemma4.1}, for every $(v_1,\ldots,v_{r-1})\in\mathcal{U}$ it holds
\[
    \limsup_{M\to\infty}\Phi_M(v_1,\ldots,v_{r-1})\geq 0.
\]
Fix $\delta>0$ and $K\in\N$. For $\N\ni M\geq K$, define $\mathcal{V}_M:=\{(v_1,\ldots,v_{r-1})\in\mathcal{U}: \Phi_M(v_1,\ldots,v_{r-1})>-\delta\}$. Then, $\{\mathcal{V}_M: M\geq K\}$ is an open cover of $\mathcal{U}$. Since $\mathcal{U}$ is compact, there exists a finite subcover, i.e. $\mathcal{U}\subset \mathcal{V}_{M_1}\cup\cdots\cup\mathcal{V}_{M_s}$, where $K\leq M_1,\ldots,M_s\in\N$. Denote $L=\max\{M_1,\ldots,M_s\}$, which depends only on $\delta$ and $K$. Obviously, $(u_1,\ldots,u_{r-1})\in\mathcal{U}$, and therefore there exists $M\in[K,L]$ s.t. $\Phi_M(u_1,\ldots,u_{r-1})>-\delta$.
This completes the proof.
\end{proof}

\subsubsection{Proof of \eqref{mhdim of sponge system} in Theorem \ref{mdim of sponge system}: upper bound}\label{mhdim:upper}
Since $\log Z_N$ is sub-additive in $N$, there exists $s>0$ satisfying that $\log_{m_1}Z_N\leq sN$ for all $N\in\N$.

\smallskip
Fix $N\in\N$. We define a probability measure $\nu_N$ on $X_{\Omega}\vert_N$ as the push-forward of $\mu_N$ under the map
\begin{equation*}
    \begin{split}
        (\Omega\vert_N)^{\N}\to X_{\Omega}\vert_N,  (\mathbf{x}_1,\ldots,\mathbf{x}_r)\mapsto \left(\sum\limits_{k=1}^{\infty} \frac{\mathbf{x}_{k,1}}{m_1^k},\ldots,\sum\limits_{k=1}^{\infty}\frac{\mathbf{x}_{k,r}}{m_r^k}\right).
    \end{split}
\end{equation*}
Recall that $X_{\Omega}\vert_N\subset[0,1]^{rN}$. Set
\begin{equation*}
    X(N)=X_{\Omega}\vert_N\times[0,1]^r\times[0,1]^r\times\cdots\subset([0,1]^r)^{\N}.
\end{equation*}
Then, $X_{\Omega}\subset X(N)$. Recall also that the metric $d$ on $([0,1]^r)^{\N}$ is defined by
\begin{equation*}
    d((\mathbf{x}_1,\ldots,\mathbf{x}_r),(\mathbf{y}_1,\ldots,\mathbf{y}_r))=\sum\limits_{k=1}^{\infty}2^{-k}\max\{|\mathbf{x}_{k,l}-\mathbf{y}_{k,l}|:1\leq l\leq r\}
\end{equation*}
Denote by $\Leb$ the Lebesgue measure on $[0,1]^r$. We define
\begin{equation*}
    \nu'_N:=\nu_N\otimes\Leb\otimes\Leb\otimes\cdots,
\end{equation*}
which is a probability measure on $X(N)$. For $N,M\in\N$ and $(\mathbf{x}_1,\ldots,\mathbf{x}_r)\in (\Omega\vert_N)^{\N}$,
set
\begin{equation*}
    Q'_{N,M}(\mathbf{x}_1,\ldots,\mathbf{x}_r):=Q_{N,M}(\mathbf{x}_1,\ldots,\mathbf{x}_r)\times[0,1]^r\times[0,1]^r\times\cdots\subset X(N).
\end{equation*}
Then, we have
\begin{equation*}
    \nu'_N(Q'_{N,M}(\mathbf{x}_1,\ldots,\mathbf{x}_r))=\nu_N(Q_{N,M}(\mathbf{x}_1,\ldots,\mathbf{x}_r))\geq\mu_N(P_{N,M}(\mathbf{x}_1,\ldots,\mathbf{x}_r)).
\end{equation*}
Fix $\varepsilon>0$ and $0<\delta<\varepsilon$ with $6\delta^{\varepsilon}<1$. Fix $K\in\N$ satisfying
\begin{equation}\label{choice of K}
    m_rm_1^{-K}<\delta/12 \text{ and } (2m_r)^{s+\varepsilon}\leq m_1^{\varepsilon K/2}.
\end{equation}
Let $L=L(\frac{\varepsilon\log m_1}{2},K)\in\N$ be the number in Lemma \ref{Lemma5.11}. Then, for any $N\in\N$ and any $(\mathbf{x}_1,\ldots,\mathbf{x}_r)\in (\Omega\vert_N)^{\N}$, there exists $M\in[K,L]$ satisfying
    \begin{equation*}
        \frac{1}{NM}\log \mu_N(P_{N,M}(\mathbf{x}_1,\ldots,\mathbf{x}_r))
        \geq-\frac{\log Z_N}{N}-\frac{\varepsilon\log m_1}{2}.
    \end{equation*}
Then,
\begin{equation*}
    \begin{split}
        \nu'_N(Q'_{N,M}(\mathbf{x}_1,\ldots,\mathbf{x}_r))
        &\geq\mu_N(P_{N,M}(\mathbf{x}_1,\ldots,\mathbf{x}_r))\\
        &\geq\exp\left\{ NM\left(-\frac{\log Z_N}{N}-\frac{\varepsilon\log m_1}{2}\right) \right\}\\
        &=\exp\left\{ -M\log m_1\left(\log_{m_1}Z_N+\frac{\varepsilon N}{2}\right) \right\}\\
        &=(m_1^{-M})^{\log_{m_1}Z_N+\frac{\varepsilon N}{2}}.
    \end{split}
\end{equation*}
Notice that $(m_1^{-M})^{\log_{m_1}Z_N+\frac{\varepsilon N}{2}}\geq (2m_rm_1^{-M})^{\log_{m_1}Z_N+\varepsilon N}$. Indeed, it is equivalent to $m_1^{\frac{\varepsilon M}{2}}\geq (2m_r)^{\frac{\log_{m_1}Z_N}{N}+\varepsilon}$, which is obvious due to \eqref{choice of K} and the fact that $\frac{\log_{m_1}Z_N}{N}\leq s$.

Choose $N_0\in\N$ with $\sum_{n>N_0}2^{-n}<m_rm_1^{-L}$.
For $\mathbf{u}=(\mathbf{u}_k)_{k\in\N}$ with $\mathbf{u}_k\in[0,1]^r$, denote $\mathbf{u}\vert_N=(\mathbf{u}_1,\ldots,\mathbf{u}_N)\in([0,1]^r)^N$. Then, for any $\mathbf{u},\mathbf{v}\in([0,1]^r)^{\N}$ and $N\geq N_0$, we have
\begin{equation}\label{distance}
    d_{N-N_0}(\mathbf{u},\mathbf{v})<\|\mathbf{u}\vert_N-\mathbf{v}\vert_N\|_{\infty}+m_rm_1^{-L}.
\end{equation}
\begin{claim}
    For any $(\mathbf{x}_1,\ldots,\mathbf{x}_r)\in (\Omega\vert_N)^{\N}$ and $N,M\in\N$ with $N>N_0$ and $M\leq L$, we have
    \begin{equation*}
        0<\diam(Q'_{N,M}(\mathbf{x}_1,\ldots,\mathbf{x}_r),d_{N-N_0})<2m_rm_1^{-M}.
    \end{equation*}
\end{claim}
\begin{proof}
    Observe that $Q'_{N,M}(\mathbf{x}_1,\ldots,\mathbf{x}_r)$ is not a single point, and thus its diameter is positive. By \eqref{distance}, we have
    \begin{equation*}
        \diam(Q'_{N,M}(\mathbf{x}_1,\ldots,\mathbf{x}_r),d_{N-N_0})<\diam(Q_{N,M}(\mathbf{x}_1,\ldots,\mathbf{x}_r),\|\cdot\|_{\infty})+m_rm_1^{-L}.
    \end{equation*}
    Note that $\diam(Q_{N,M}(\mathbf{x}_1,\ldots,\mathbf{x}_r),\|\cdot\|_{\infty})\leq m_rm_1^{-M}$, as shown in Section \ref{Calculation of metric mean dimension}. Then,
    \begin{equation*}
        \diam(Q'_{N,M}(\mathbf{x}_1,\ldots,\mathbf{x}_r),d_{N-N_0})<m_rm_1^{-M}+m_rm_1^{-L}\leq 2m_rm_1^{-M}
    \end{equation*}
    since $M\leq L$.
\end{proof}
Notice that $2m_rm_1^{-M}<\delta/6$ since we have assumed that $m_rm_1^{-K}<\delta/12$ in \eqref{choice of K}.
Up to now, we have proven that for any $N>N_0$ and any $(\mathbf{x}_1,\ldots,\mathbf{x}_r)\in (\Omega\vert_N)^{\N}$, there exists $M\in[K,L]$ such that
\begin{equation*}
    \begin{split}
        &0<\diam(Q'_{N,M}(\mathbf{x}_1,\ldots,\mathbf{x}_r),d_{N-N_0})<\frac{\delta}{6},\\
        &\nu'_N(Q'_{N,M}(\mathbf{x}_1,\ldots,\mathbf{x}_r))\geq (\diam(Q'_{N,M}(\mathbf{x}_1,\ldots,\mathbf{x}_r),d_{N-N_0}))^{\log_{m_1}Z_N+\varepsilon N}.
    \end{split}
\end{equation*}
Since $6\delta^{\varepsilon}<1$, by applying Lemma \ref{Hausdorff dim estimate} to the space $(X(N),d_{N-N_0})$, we obtain that
\begin{equation*}
    \hdim(X(N),d_{N-N_0},\delta)\leq (1+\varepsilon)(\log_{m_1}Z_N+\varepsilon N)
\end{equation*}
for all $N>N_0$. Since $\delta<\varepsilon$ and $X_{\Omega}\subset X(N)$, we further have
\begin{equation*}
    \frac{\hdim(X_{\Omega},d_{N-N_0},\varepsilon)}{N}\leq (1+\varepsilon)\left(\frac{\log_{m_1}Z_N}{N}+\varepsilon \right)
\end{equation*}
for all $N>N_0$. Letting $N\to\infty$, by Lemma \ref{Formula for the weighted topological entropy} we have
\begin{equation*}
\begin{split}
    \limsup_{N\to\infty}\frac{\hdim(X_{\Omega},d_{N},\varepsilon)}{N}
    &\leq (1+\varepsilon)\left(\frac{1}{\log m_1}\lim_{N\to\infty}\frac{\log Z_N}{N}+\varepsilon \right)\\
    &=(1+\varepsilon)\left(\frac{\htop^{\mathbf{a}}(\{(\pi_i(\Omega),\sigma)\}_{i=1}^r,\{\tau_i\}_{i=1}^{r-1})}{\log m_1}+\varepsilon \right).
\end{split}
\end{equation*}
Letting $\varepsilon\to0$, we finally obtain the desired upper bound, i.e.
\begin{equation}\label{upper bound}
    \umhdim(X_{\Omega},\sigma,d)\leq \frac{\htop^{\mathbf{a}}(\{(\pi_i(\Omega),\sigma)\}_{i=1}^r,\{\tau_i\}_{i=1}^{r-1})}{\log m_1}.
\end{equation}

\subsubsection{Proof of \eqref{mhdim of sponge system} in Theorem \ref{mdim of sponge system}: lower bound}\label{mhdim:lower}
\begin{lemma}\label{Lemma5.12}
    For any $\delta>0$, there exists $M_1=M_1(\delta)\in\N$ such that for any $N\in\N$ there is a Borel subset $R(\delta,N)\subset(\Omega\vert_N)^{\N}$ satisfying the following two conditions.
    \begin{itemize}
        \item $\mu_N(R(\delta,N))\geq\frac{1}{2}$.
        \item For any $M\geq M_1$ and $(\mathbf{x}_1,\ldots,\mathbf{x}_r)\in R(\delta,N)$ we have
        \begin{equation*}
            \left|\frac{1}{NM}\log \mu_N(P_{N,M}(\mathbf{x}_1,\ldots,\mathbf{x}_r))+\frac{\log Z_N}{N}\right|\leq\delta.
        \end{equation*}
    \end{itemize}
\end{lemma}
\begin{proof}
    Fix $1\leq i\leq r-1$. For $k\in\N$, we define $\xi_{k,i}:(\Omega\vert_N)^{\N}\to\R$ by
    \begin{equation*}
        \xi_{k,i}=\xi_{k,i}(\mathbf{x}_1,\ldots,\mathbf{x}_r):=\frac{\log Z_N(\mathbf{x}_{k,1},\ldots,\mathbf{x}_{k,i})}{N}.
    \end{equation*}
    The random variables $\xi_{k,i},k\in\N$ are independent and identically distributed with respect to the measure $\mu_N=(f_N)^{\otimes\N}$. Moreover, \eqref{eq:ZN uniform estimate} yields that $0\leq \xi_{k,i}\leq \sum_{j=i+1}^r\log m_j\leq \sum_{j=2}^r\log m_j$, and thus its mean and variance satisfy
    \begin{equation*}
        0\leq \mathbb{E}(\xi_{k,i}) \leq \sum_{j=2}^r\log m_j, \mathbb{V}(\xi_{k,i}) \leq \left(\sum_{j=2}^r\log m_j\right)^2.
    \end{equation*}
    Take $c_i=\log m_1/\log m_i$ and $\widehat{c}_i=\log m_1/\log m_{i+1}$ for each $1\leq i\leq r-1$. Then, $\lfloor c_iM \rfloor=L_i(M)$ and $\lfloor \widehat{c}_iM \rfloor=L_{i+1}(M)$. We apply Lemma \ref{Lemma5.9} to $\{\xi_{k,i}\}_{k\in\N,1\leq i\leq r-1}$, and thus there exists $M_1=M_1(\delta)\in\N$ such that the set
    \begin{equation*}
        R(\delta,N):=\left\{(\mathbf{x}_1,\ldots,\mathbf{x}_r)\in(\Omega\vert_N)^{\N} : \sup_{M\geq M_1}\sum_{i=1}^{r-1}\left|  \frac{\sum_{k=1}^{L_i(M)}\xi_{k,i}}{M\frac{\log m_1}{\log m_i}}
        -\frac{\sum_{k=1}^{L_{i+1}(M)}\xi_{k,i}}{M\frac{\log m_1}{\log m_{i+1}}}  \right|\leq\delta \right\}
    \end{equation*}
    satisfies that $\mu_N(R(\delta,N))\geq\frac{1}{2}$. Recall that
    \begin{equation*}
        u_i(L)=\sum_{k=1}^{L}\frac{\log Z_N(\mathbf{x}_{k,1},\ldots,\mathbf{x}_{k,i})}{N}=\sum_{k=1}^{L}\xi_{k,i}
    \end{equation*}
    and $a_i=\log m_{r-i}/\log m_{r-i+1}$. By \eqref{(5.8)},
    \begin{equation*}
    \begin{split}
        &\frac{1}{NM}\log \mu_N(P_{N,M}(\mathbf{x}_1,\ldots,\mathbf{x}_r))+\frac{\log Z_N}{N}\\
        &=\sum_{i=1}^{r-1}\left(a_{r-i}\frac{u_i(L_i(M))}{M}-\frac{u_{i}(L_{i+1}(M))}{M}\right)\\
    &=\sum_{i=1}^{r-1}\frac{\log m_1}{\log m_{i+1}} \left( \frac{u_i(L_i(M))}{M\frac{\log m_1}{\log m_i}} -\frac{u_i(L_{i+1}(M))}{M\frac{\log m_1}{\log m_{i+1}}} \right)\\
    &=\sum_{i=1}^{r-1}\frac{\log m_1}{\log m_{i+1}} \left( \frac{\sum_{k=1}^{L_i(M)}\xi_{k,i}}{M\frac{\log m_1}{\log m_i}} -\frac{\sum_{k=1}^{L_{i+1}(M)}\xi_{k,i}}{M\frac{\log m_1}{\log m_{i+1}}} \right)\\
    \end{split}
    \end{equation*}
    and therefore
    \begin{equation*}
        \left|\frac{1}{NM}\log \mu_N(P_{N,M}(\mathbf{x}_1,\ldots,\mathbf{x}_r))+\frac{\log Z_N}{N}\right|
        \leq\sum_{i=1}^{r-1} \left|   \frac{\sum_{k=1}^{L_i(M)}\xi_{k,i}}{M\frac{\log m_1}{\log m_i}}
        -\frac{\sum_{k=1}^{L_{i+1}(M)}\xi_{k,i}}{M\frac{\log m_1}{\log m_{i+1}}}  \right|.
    \end{equation*}
    Furthermore, for any $M\geq M_1$ and $(\mathbf{x}_1,\ldots,\mathbf{x}_r)\in R(\delta,N)$,
    \begin{equation*}
            \left|\frac{1}{NM}\log \mu_N(P_{N,M}(\mathbf{x}_1,\ldots,\mathbf{x}_r))+\frac{\log Z_N}{N}\right|\leq\delta.
    \end{equation*}
\end{proof}

Recall that we have defined the approximate cube of level $M\in\N$ centred at $(\mathbf{x}_1,\ldots,\mathbf{x}_r)\in (\Omega\vert_N)^{\N}$ by
\begin{equation*}
\begin{split}
     & Q_{N,M}(\mathbf{x}_1,\ldots,\mathbf{x}_r) \\
     & =
   \left\{\left(\sum\limits_{k=1}^{\infty} \frac{\mathbf{y}_{k,1}}{m_1^k},\ldots,\sum\limits_{k=1}^{\infty}\frac{\mathbf{y}_{k,r}}{m_r^k}\right)  \middle|\,
    \parbox{2.5in}{\centering $(\mathbf{y}_{k,1},\ldots,\mathbf{y}_{k,r})\in\Omega\vert_N$, $\forall k\in\N$ with \\
    $\mathbf{y}_{k,i}=\mathbf{x}_{k,i}$ for all $1\leq k\leq L_i(M)$ \\ and $1\leq i \leq r$} \right\},
\end{split}
\end{equation*}
which is the image of the set $P_{N,M}(\mathbf{x}_1,\ldots,\mathbf{x}_r)$ under the map
\begin{equation*}
    \begin{split}
        (\Omega\vert_N)^{\N}\to X_{\Omega}\vert_N,  (\mathbf{x}_1,\ldots,\mathbf{x}_r)\mapsto \left(\sum\limits_{k=1}^{\infty} \frac{\mathbf{x}_{k,1}}{m_1^k},\ldots,\sum\limits_{k=1}^{\infty}\frac{\mathbf{x}_{k,r}}{m_r^k}\right).
    \end{split}
\end{equation*}
It is direct to observe that for $(\mathbf{x}_1,\ldots,\mathbf{x}_r)=((\mathbf{x}_{k,1})_{k\in\N},\ldots,(\mathbf{x}_{k,r})_{k\in\N})$ and $(\mathbf{y}_1,\ldots,\\ \mathbf{y}_r)=((\mathbf{y}_{k,1})_{k\in\N},\ldots,(\mathbf{y}_{k,r})_{k\in\N})$ in $(\Omega\vert_N)^{\N}$, the following three conditions are equivalent to each other:
\begin{itemize}
    \item $P_{N,M}(\mathbf{x}_1,\ldots,\mathbf{x}_r)=P_{N,M}(\mathbf{y}_1,\ldots,\mathbf{y}_r)$.
    \item $Q_{N,M}(\mathbf{x}_1,\ldots,\mathbf{x}_r)=Q_{N,M}(\mathbf{y}_1,\ldots,\mathbf{y}_r)$.
    \item $(\mathbf{x}_{1,1},\ldots,\mathbf{x}_{L_1(M),1},\ldots,\mathbf{x}_{1,r},\ldots,\mathbf{x}_{L_r(M),r})\\
        =(\mathbf{y}_{1,1},\ldots,\mathbf{y}_{L_1(M),1},\ldots,\mathbf{y}_{1,r},\ldots,\mathbf{y}_{L_r(M),r})$.
\end{itemize}
Indeed, both the sets $P_{N,M}(\mathbf{x}_1,\ldots,\mathbf{x}_r)$ and $Q_{N,M}(\mathbf{x}_1,\ldots,\mathbf{x}_r)$ depend only on the coordinates $\mathbf{x}_{1,1},\ldots,\mathbf{x}_{L_1(M),1},\ldots,\mathbf{x}_{1,r},\ldots,\mathbf{x}_{L_r(M),r}$.

\smallskip
For $E,F\subset X_{\Omega}\vert_N$, denote
\begin{equation*}
    \dist_{\infty}(E,F)=\inf_{\substack{\mathbf{x}\in E \\ \mathbf{y}\in F}}\|\mathbf{x}-\mathbf{y}\|_{\infty},
\end{equation*}
where $\|\cdot\|_{\infty}$ is the $\ell^{\infty}$-distance in $\R^{rN}$.
\begin{lemma}\label{Lemma5.13}
    Fix $N,M\in\N$. Let $(\mathbf{x}_1^{(1)},\ldots,\mathbf{x}_r^{(1)}),\ldots,(\mathbf{x}_1^{(l)},\ldots,\mathbf{x}_r^{(l)})\in(\Omega\vert_N)^{\N}$ with $l\geq 2^{rN}+1$. Suppose that
    \begin{equation*}
        P_{N,M}(\mathbf{x}_1^{(i)},\ldots,\mathbf{x}_r^{(i)}) \neq P_{N,M}(\mathbf{x}_1^{(j)},\ldots,\mathbf{x}_r^{(j)}) \text{ for all } i\neq j.
    \end{equation*}
    Then, there exist $i$ and $j$ for which
    \begin{equation*}
        \dist_{\infty}(Q_{N,M}(\mathbf{x}_1^{(i)},\ldots,\mathbf{x}_r^{(i)}),Q_{N,M}(\mathbf{x}_1^{(j)},\ldots,\mathbf{x}_r^{(j)})) \geq m_1^{-M}.
    \end{equation*}
\end{lemma}

\begin{proof}
    Since $P_{N,M}(\mathbf{x}_1^{(i)},\ldots,\mathbf{x}_r^{(i)}) \neq P_{N,M}(\mathbf{x}_1^{(j)},\ldots,\mathbf{x}_r^{(j)})$, we have
    \begin{equation*}
       \begin{aligned}  &(\mathbf{x}_{1,1}^{(i)},\ldots,\mathbf{x}_{L_1(M),1}^{(i)},\ldots,\mathbf{x}_{1,r}^{(i)},\ldots,\mathbf{x}_{L_r(M),r}^{(i)})\\
        & \neq
        (\mathbf{x}_{1,1}^{(j)},\ldots,\mathbf{x}_{L_1(M),1}^{(j)},\ldots,\mathbf{x}_{1,r}^{(j)},\ldots,\mathbf{x}_{L_r(M),r}^{(j)}).
    \end{aligned} \end{equation*}
    Furthermore, since $l\geq 2^{rN}+1$, there are $i$ and $j$ for which there exists at least one $q\in\{1,2,\ldots,r\}$ such that
    \begin{equation}\label{distance between distinct points}
        \left\|\sum_{k=1}^{L_q(M)}\frac{\mathbf{x}_{k,q}^{(i)}}{m_q^k}
        -\sum_{k=1}^{L_q(M)}\frac{\mathbf{x}_{k,q}^{(j)}}{m_q^k}\right\|_{\infty}
        \geq 2m_q^{-L_q(M)}.
    \end{equation}
    Notice that
    \begin{equation*}
       \begin{aligned}  &Q_{N,M}(\mathbf{x}_1,\ldots,\mathbf{x}_r)\\
       &\subset \left( \sum\limits_{k=1}^{L_1(M)} \frac{\mathbf{x}_{k,1}}{m_1^k},\ldots,\sum\limits_{k=1}^{L_r(M)}\frac{\mathbf{x}_{k,r}}{m_r^k} \right)+
        [0,m_1^{-L_1(M)}]^N\times\cdots\times[0,m_r^{-L_r(M)}]^N
     \end{aligned}\end{equation*}
    for all $(\mathbf{x}_1,\ldots,\mathbf{x}_r)\in(\Omega\vert_N)^{\N}$. Thus, inequality \eqref{distance between distinct points} implies that
    \begin{equation*}
        \dist_{\infty}(Q_{N,M}(\mathbf{x}_1^{(i)},\ldots,\mathbf{x}_r^{(i)}),Q_{N,M}(\mathbf{x}_1^{(j)},\ldots,\mathbf{x}_r^{(j)})) \geq m_q^{-L_q(M)}\geq m_1^{-M}
    \end{equation*}
    and we obtain the desired result.
\end{proof}

As in the proof of the upper bound of \eqref{mhdim of sponge system}, choose $s>0$ with $\log_{m_1}Z_N\leq sN$ for all $N\in\N$. Fix $N\in\N$ and $\delta>0$. Without loss of generality, assume that $\Omega\vert_N$ contains at least two points.
Let $M_1=M_1(\frac{\delta\log m_1}{2})$ be introduced in Lemma \ref{Lemma5.12}. Then, there exists a Borel subset $R=R(\frac{\delta\log m_1}{2},N)\subset(\Omega\vert_N)^{\N}$ satisfying $\mu_N(R)\geq\frac{1}{2}$, and for any $M\geq M_1$ and $(\mathbf{x}_1,\ldots,\mathbf{x}_r)\in R$, we have
        \begin{equation*}
            \left|\frac{1}{NM}\log \mu_N(P_{N,M}(\mathbf{x}_1,\ldots,\mathbf{x}_r))+\frac{\log Z_N}{N}\right|\leq\frac{\delta\log m_1}{2}.
        \end{equation*}
Fix $\varepsilon>0$ with
\begin{equation}\label{(5.11)}
    \varepsilon<m_1^{-M_1-1} \text{ and }  2^rm_1^s\varepsilon^{\delta/2}<\frac{1}{2}.
\end{equation}

\smallskip
Let $X_{\Omega}\vert_N=\bigcup_{j=1}^{\infty}E_j$ be a cover with $\diam(E_j,\|\cdot\|_{\infty})<\varepsilon$ for each $j\geq 1$. We write $D(E_j):=\diam(E_j,\|\cdot\|_{\infty})$ for brevity. We now prove that
\begin{equation*}
    \sum_{j=1}^{\infty}D(E_j)^{\log Z_N-\delta N}>1.
\end{equation*}
For each $E_j$, let $\widehat{L}_j$ be the unique integer satisfying
\begin{equation*}
    m_1^{-\widehat{L}_j-1}\leq D(E_j) <m_1^{-\widehat{L}_j}.
\end{equation*}
Then, $\widehat{L}_j\geq M_1$ since we have assumed that $\varepsilon<m_1^{-M_1-1}$ in \eqref{(5.11)}.
Let
\begin{equation*}
    \mathcal{C}_j=\{P_{N,\widehat{L}_j}(\mathbf{x}_1,\ldots,\mathbf{x}_r): (\mathbf{x}_1,\ldots,\mathbf{x}_r)\in R \text{ with } Q_{N,\widehat{L}_j}(\mathbf{x}_1,\ldots,\mathbf{x}_r)\cap E_j\neq\emptyset\}.
\end{equation*}
Then, for $P_{N,\widehat{L}_j}(\mathbf{x}_1,\ldots,\mathbf{x}_r)$ and $P_{N,\widehat{L}_j}(\mathbf{y}_1,\ldots,\mathbf{y}_r)$ in $\mathcal{C}_j$, we have
\begin{equation*}
    \dist_{\infty}(Q_{N,\widehat{L}_j}(\mathbf{x}_1,\ldots,\mathbf{x}_r),Q_{N,\widehat{L}_j}(\mathbf{y}_1,\ldots,\mathbf{y}_r))\leq D(E_j)<\varepsilon<m_1^{-\widehat{L}_j}.
\end{equation*}
By Lemma \ref{Lemma5.13}, we further have $|\mathcal{C}_j|\leq 2^{rN}$.

\smallskip
Observe that
\begin{equation*}
    R\subset\bigcup_{j=1}^{\infty}\bigcup_{P\in\mathcal{C}_j}P
\end{equation*}
and thus
\begin{equation}\label{(5.12)}
    \frac{1}{2}\leq \mu_N(R)\leq \sum_{j=1}^{\infty}\sum_{P\in\mathcal{C}_j} \mu_N(P).
\end{equation}
Since $\widehat{L}_j\geq M_1$, for every $P=P_{N,\widehat{L}_j}(\mathbf{x}_1,\ldots,\mathbf{x}_r)\in\mathcal{C}_j$ with $(\mathbf{x}_1,\ldots,\mathbf{x}_r)\in R$, we have
\begin{equation*}
    \frac{1}{N \widehat{L}_j}\log \mu_N(P)\leq
    -\frac{\log Z_N}{N}+\frac{\delta\log m_1}{2}
\end{equation*}
and therefore
\begin{equation*}
    \begin{split}
        \mu_N(P)&\leq\exp\left\{ \widehat{L}_j\left( -\log Z_N+\frac{N\delta\log m_1}{2}\right) \right\}\\
        &=\exp\left\{ \widehat{L}_j\log m_1\left( -\log_{m_1} Z_N+\frac{N\delta}{2}\right) \right\}\\
        &=m_1^{-\widehat{L}_j\left(\log_{m_1} Z_N-\frac{N\delta}{2}\right)}\\
        &\leq (m_1 D(E_j))^{\log_{m_1} Z_N-\frac{N\delta}{2}}.
    \end{split}
\end{equation*}
Then, by \eqref{(5.12)} we have
\begin{align*}
        \frac{1}{2}&\leq \sum_{j=1}^{\infty}\sum_{P\in\mathcal{C}_j} (m_1 D(E_j))^{\log_{m_1} Z_N-\frac{N\delta}{2}} \\
        &\leq \sum_{j=1}^{\infty} 2^{rN}(m_1 D(E_j))^{\log_{m_1} Z_N-\frac{N\delta}{2}}  \text{ (using }  |\mathcal{C}_j|\leq 2^{rN})\\
        &<\sum_{j=1}^{\infty} 2^{rN} m_1^{\log_{m_1} Z_N} D(E_j)^{\log_{m_1} Z_N-\frac{N\delta}{2}}\\
        &\leq \sum_{j=1}^{\infty} 2^{rN} m_1^{sN} D(E_j)^{\log_{m_1} Z_N-\frac{N\delta}{2}} \text{ (using } \log_{m_1} Z_N \leq sN)\\
        &=\sum_{j=1}^{\infty} 2^{rN} m_1^{sN}D(E_j)^{\frac{N\delta}{2}}\cdot D(E_j)^{\log_{m_1} Z_N-N\delta}\\
        &<\sum_{j=1}^{\infty} 2^{rN} m_1^{sN}\varepsilon^{\frac{N\delta}{2}}\cdot D(E_j)^{\log_{m_1} Z_N-N\delta} \text{ (using } D(E_j)<\varepsilon)\\
        &\overset{\eqref{(5.11)}}{<}\sum_{j=1}^{\infty}\frac{1}{2^N} D(E_j)^{\log_{m_1} Z_N-N\delta}
\end{align*}
and hence
\begin{equation*}
    \sum_{j=1}^{\infty} D(E_j)^{\log_{m_1} Z_N-N\delta}>1,
\end{equation*}
which yields that $\hdim(X_{\Omega}\vert_N,\|\cdot\|_{\infty},\varepsilon)\geq \log_{m_1} Z_N-\delta N$. Notice that the projection $(X_{\Omega},d_N)\to (X_{\Omega}\vert_N,\|\cdot\|_{\infty})$ defined by
\begin{equation*}
    ([0,1]^{\N})^r \to [0,1]^{rN},
    ((\mathbf{y}_{k,1})_{k\in\N},\ldots,(\mathbf{y}_{k,r})_{k\in\N}) \mapsto ((\mathbf{y}_{k,1})_{1\leq k\leq N},\ldots,(\mathbf{y}_{k,r})_{1\leq k\leq N})
\end{equation*}
is $1$-Lipschitz, and therefore
\begin{equation*}
   \frac{\hdim(X_{\Omega},d_N,\varepsilon)}{N} \geq \frac{\hdim(X_{\Omega}\vert_N,\|\cdot\|_{\infty},\varepsilon)}{N} \geq \frac{\log_{m_1} Z_N}{N}-\delta.
\end{equation*}
We have used the fact that the Hausdorff dimension does not increase under Lipschitz maps, see e.g. \cite[Proposition 2.8 (iv)]{Rob11}.
Letting $N\to\infty$, by Lemma \ref{Formula for the weighted topological entropy} we have
\begin{equation*}
    \begin{split}
        \liminf_{N\to\infty}\frac{\hdim(X_{\Omega},d_N,\varepsilon)}{N}
        &\geq \lim_{N\to\infty}\frac{\log_{m_1} Z_N}{N}-\delta\\
        &=\frac{1}{\log m_1}\lim_{N\to\infty}\frac{\log Z_N}{N}-\delta  \\
        &=\frac{\htop^{\mathbf{a}}(\{(\pi_i(\Omega),\sigma)\}_{i=1}^r,\{\tau_i\}_{i=1}^{r-1})}{\log m_1}-\delta.
    \end{split}
\end{equation*}
Letting $\varepsilon\to0$ and $\delta\to0$, we finally obtain that
\begin{equation}\label{lower bound}
    \lmhdim(X_{\Omega},\sigma,d)\geq \frac{\htop^{\mathbf{a}}(\{(\pi_i(\Omega),\sigma)\}_{i=1}^r,\{\tau_i\}_{i=1}^{r-1})}{\log m_1}.
\end{equation}

Combining \eqref{upper bound} and \eqref{lower bound}, we have proven that
\begin{equation*}
    \mhdim(X_{\Omega},\sigma,d)= \frac{\htop^{\mathbf{a}}(\{(\pi_i(\Omega),\sigma)\}_{i=1}^r,\{\tau_i\}_{i=1}^{r-1})}{\log m_1}.
\end{equation*}
The proof is complete.

\section*{Acknowledgments}
The author would like to thank the anonymous referees for their careful reading and many helpful suggestions which improved the manuscript.









\medskip
Received August 12, 2025; revised April 26, 2026.
\medskip

\end{document}